\documentclass[a4paper, 10pt, notitlepage]{article}

\usepackage{amsthm, amsmath, amssymb, latexsym}
\usepackage{mathrsfs}
\usepackage[multiple]{footmisc}
\usepackage{graphicx}
\usepackage[ruled,vlined]{algorithm2e}

\theoremstyle{plain}
\newtheorem{theorem}{Theorem}[section]
\newtheorem{lemma}[theorem]{Lemma}
\newtheorem{corollary}[theorem]{Corollary}
\newtheorem{proposition}[theorem]{Proposition}

\theoremstyle{definition}
\newtheorem{definition}[theorem]{Definition}
\newtheorem{example}[theorem]{Example}

\theoremstyle{remark}
\newtheorem{remark}{Remark}[section]

\usepackage{chngcntr}
\counterwithin{theorem}{section}	
\counterwithin{remark}{section}	

\DeclareMathOperator{\co}{co}
\DeclareMathOperator{\cl}{cl}
\DeclareMathOperator{\relint}{relint}
\DeclareMathOperator*{\argmin}{arg\,min}

\DeclareMathOperator{\interior}{int}
\DeclareMathOperator{\dist}{dist}

\author{Dolgopolik M.V.\footnote{Saint Petersburg State University, Saint Petersburg, Russia}
\footnote{Institute for Problems in Mechanical Engineering of the Russian Academy of Sciences, Saint Petersburg,
Russia}}
\title{The method of codifferential descent for convex and global piecewise affine optimization}

\begin{document}

\maketitle

\begin{abstract}
The class of nonsmooth codifferentiable functions was introduced by professor V.F.~Demyanov in the late 1980s. He also
proposed a method for minimizing these functions called the method of codifferential descent (MCD). However, until now
almost no theoretical results on the performance of this method on particular classes of nonsmooth optimization problems
were known. In the first part of the paper, we study the performance of the method of codifferential descent on a class
of nonsmooth convex functions satisfying some regularity assumptions, which in the smooth case are reduced to the
Lipschitz continuity of the gradient. We prove that in this case the MCD has the iteration complexity bound
$\mathcal{O}(1 / \varepsilon)$. In the second part of the paper we obtain new global optimality conditions for piecewise
affine functions in terms of codifferentials. With the use of these conditions we propose a modification of the MCD for
minimizing piecewise affine functions (called the method of global codifferential descent) that does not use line
search, and discards those ``pieces'' of the objective functions that are no longer useful for the optimization
process. Then we prove that the MCD as well as its modification proposed in the article find a point of global minimum
of a nonconvex piecewise affine function in a finite number of steps.
\end{abstract}

\section{Introduction}

An interesting approach to the analysis of nonsmooth functions based on the use of continuous approximations called
codifferentials was proposed by Demyanov in \cite{Demyanov_InCollection_1988,Demyanov1988,Demyanov1989}. He
developed the codifferential calculus \cite{DemRub_book} (see~\cite{Zaffaroni,Dolgopolik_CodiffCalc,Dolgopolik_Gen} for
its extensions and generalizations), and proposed a method for minimizing codifferentiable functions called the method
of codifferential descent (MCD). This method was applied to some problems of cluster analysis \cite{DemBagRub},
computational geometry \cite{TamasyanChumacov,TamasyanChumacov_Conf}, calculus of variations
\cite{DemyanovTamasyan2011,DemyanovTamasyan2014} and optimal control problems \cite{Fominyh,Fominyh_OptimLetters}.
Hybrid methods for solving convex and DC (difference-of-convex) optimization problems combining the ideas of bundle
methods and the MCD were proposed in \cite{BagGanUgonTor,BagUgon,TorBagKar}. A comprehensive convergence analysis of the
MCD and some of its modifications was presented in the recent paper \cite{Dolgopolik_MCD}. However, almost nothing is
known about the global performance of the MCD on particular classes of nonsmooth optimization problems apart from some
results of numerical experiments.

The main goal of this article is to analyse the overall performance of the method of codifferential descent in two
tractable cases. Namely, the first part of the paper is devoted to the analysis of this method in the convex case.
Some of the most popular black-box methods of convex optimization are subgradient methods
\cite{Nesterov,Nesterov2015,HuSimYang2015,Neumaier2016,YangLin2018} and bundle methods
\cite{HiriartUrruty,Makela2002,Kiwiel2006,HouSun2008,OliveiraSolodov2016}. However, these methods are relatively slow in
the general case, since they require $\mathcal{O}(1 / \varepsilon^2)$ iterations to find an $\varepsilon$-optimal
solution \cite{Nesterov}. In the case when a certain information about the structure of the optimization problem under
consideration is known, one can devise significantly faster methods (see, e.g., \cite{Nesterov2005,Nesterov2013}). In
this article, we demonstrate that under some natural regularity assumptions the method of codifferential descent finds
an $\varepsilon$-optimal solution in at most $\mathcal{O}(1 / \varepsilon)$ iterations, which is better than the
iteration complexity bound for subgradient methods, despite the fact that the MCD is also a black-box method. On the
other hand, it should be noted that the MCD utilises an oracle that provides significantly more information about the
objective function than the one used by subgradient and bundle methods. Thus, in a sense, the MCD trades off the
complexity of each call of the oracle for the better rate of convergence in comparison with subgradient methods.

The second part of the paper is devoted to the analysis of the method of codifferential descent in the piecewise affine
case. As it was demonstrated via numerical simulation in \cite{DemBagRub}, the MCD ``jumps over'' some points of local
minimum of nonsmooth functions, and in some applications it is capable of finding a global minimizer of the objective
function in spite of the fact that the MCD is a black-box local search method. To understand a reason behind this
phenomenon we derive new global optimality conditions for piecewise affine functions in terms of codifferentials, which
are significantly different from the ones obtained by Polyakova \cite{Polyakova} or from the standard global optimality
conditions for DC (difference-of-convex) optimization problems \cite{HiriartUrruty_GlobalDC,HiriartUrruty_GlobalDC_2}.
It turns out that new conditions for global optimality are implicitly incorporated into the MCD. With the use of these
conditions we propose a modification of the MCD for minimizing piecewise affine functions that, unlike the original
method, does not use line search, and allows one to avoid unnecessary computations by discarding those ``pieces''
of the objective function that no longer provide useful information about the global behaviour of this function. Then
we prove that the modified MCD as well as the MCD itself find a point of global minimum of a piecewise affine function
in a finite number of steps, thus giving a first theoretical explanation for the ability of the MCD to find a
\textit{globally} optimal solution in some applications.

The paper is organized as follows. In Section~\ref{Sect_ConvexCase} some new natural regularity assumptions on nonsmooth
convex functions are introduced, and the performance of the MCD on the class of nonsmooth convex functions satisfying
these assumptions is analysed. New necessary and sufficient global optimality conditions for piecewise affine functions
in terms of codifferentials are obtained in Section~\ref{Sect_PiecewiseAffineCase}. We utilise these conditions in order
to propose a modification of the MCD, and to prove that this modification as well as the original method find a point of
global minimum of a piecewise affine function in a finite number of steps. Finally, for reader's convenience, some
basic definitions and results from the codifferential calculus are given in Section~\ref{Sect_Preliminaries}.

\section{Preliminaries}
\label{Sect_Preliminaries}

Let $\mathcal{H}$ be a real Hilbert space, and $U$ be a neighbourhood of a point $x \in \mathcal{H}$. Recall that a
function $f \colon U \to \mathbb{R}$ is called \textit{codifferentiable} at $x$, if there exist weakly compact convex
sets $\underline{d} f(x), \overline{d} f(x) \subset \mathbb{R} \times \mathcal{H}$ such that for any 
$\Delta x \in \mathcal{H}$ one has
\begin{align*}
  \lim_{\alpha \to +0} \frac{1}{\alpha} \Big| f(x + \alpha \Delta x) - f(x) 
  &- \max_{(a, v) \in \underline{d} f(x)} \big( a + \langle v, \Delta x \rangle \big) \\
  &- \min_{(b, w) \in \overline{d} f(x)} \big( b + \langle w, \Delta x \rangle \big) \Big| 
  = 0,
\end{align*}
and
\begin{equation} \label{CodiffApproxZeroAtZero}
  \max_{(a, v) \in \underline{d} f(x)} a + \min_{(b, w) \in \overline{d} f(x)} b = 0.
\end{equation}
Here $\langle \cdot, \cdot \rangle$ is the inner product in $\mathcal{H}$, and we suppose that the space 
$\mathbb{R} \times \mathcal{H}$ is endowed with the norm $\| (a, v) \|^2 = a^2 + \| v \|_{\mathcal{H}}^2$ 
for any $(a, v) \in \mathbb{R} \times \mathcal{H}$. The pair $D f(x) = [\underline{d} f(x), \overline{d} f(x)]$ is
called a \textit{codifferential} of $f$ at $x$, the set $\underline{d} f(x)$ is called a \textit{hypodifferential} of
$f$ at $x$, while the set $\overline{d} f(x)$ is referred to as a \textit{hyperdifferential} of $f$ at $x$. Let us note
that the function $f$ is codifferentiable at $x$ if and only if its increment $f(x + \Delta x) - f(x)$ can be locally
approximated by the difference of two convex functions, i.e. by a DC function (see \cite[Example~3.10]{Dolgopolik_Gen}
for more details). Hence, in particular, any function that can be represented as the difference of convex functions is
codifferentiable.

It is easy to see that a codifferential of $f$ at $x$ is not unique. Therefore, it seems natural to single out
a codifferential of $f$ at $x$ that has some useful additional properties. At first, let us note that without loss of
generality \cite{DemRub_book,Dolgopolik_MCD} one can suppose that
\begin{equation} \label{CodiffApprox_CorrectSigns}
  \max_{(a, v) \in \underline{d} f(x)} a = \min_{(b, w) \in \overline{d} f(x)} b = 0
\end{equation}
(cf.~\eqref{CodiffApproxZeroAtZero}). At second, recall that $f$ is said to be \textit{continuously} codifferentiable at
$x$, if $f$ is codifferentiable at every point in a neighbourhood of $x$, and there exists a codifferential mapping 
$D f(\cdot) = [\underline{d} f(\cdot), \overline{d} f(\cdot)]$ defined in a neighbourhood of $x$, and such that 
the multifunctions $\underline{d} f(\cdot)$ and $\overline{d} f(\cdot)$ are Hausdorff continuous at $x$. This
codifferential mapping $D f(\cdot)$ is called continuous at $x$. Similarly, a function 
$f \colon \mathcal{H} \to \mathbb{R}$ is called continuously codifferentiable on a set $A \subset \mathcal{H}$, if $f$
is codifferentiable at every point $x \in A$, and the exists a \textit{continuous} codifferential mapping $D f(\cdot)$
defined on $A$, i.e. a codifferential mapping $D f(\cdot)$ such that the corresponding multifunctions 
$\underline{d} f \colon A \rightrightarrows \mathbb{R} \times \mathcal{H}$ and
$\overline{d} f \colon A \rightrightarrows \mathbb{R} \times \mathcal{H}$ are Hausdorff continuous on $A$. Let us note
that the set of all those nonsmooth functions that are continuously codifferentiable on a given convex set $A$ is
closed under all standard algebraic operations, the pointwise maximum and minimum of finite families of functions, as
well as the composition with smooth functions. Furthermore, there exists simple and well-developed codifferential
calculus \cite{DemRub_book,Dolgopolik_Gen,Dolgopolik_MCD}.

One can check that if a function $f \colon U \to \mathbb{R}$ is codifferentiable at $x$, then $f$ is directionally
differentiable at $x$, and the standard necessary condition for a minimum $f'(x, \cdot) \ge 0$ is satisfied 
if and only if
\begin{equation} \label{NessOptCond}
  0 \in \underline{d} f(x) + \{ (0, w) \} \quad \forall (0, w) \in \overline{d} f(x)
\end{equation}
(see~\cite{DemRub_book,Dolgopolik_MCD}). Here $f'(x, h)$ is the directional derivative of $f$ at $x$ in the direction
$h$. A point $x$ satisfying optimality condition \eqref{NessOptCond} is called an \textit{inf-stationary} point of the
function $f$. Note that the definition of inf-stationary point is independent of the choice of a codifferential, since
the optimality condition $f'(x, \cdot) \ge 0$ is invariant with respect to the choice of a codifferential.

One can utilise optimality condition \eqref{NessOptCond} to design a numerical method for minimizing codifferentiable
functions called \textit{the method of codifferential descent} \cite{DemRub_book,Dolgopolik_MCD}. Let a function 
$f \colon \mathcal{H} \to \mathbb{R}$ be codifferentiable (i.e. codifferentiable on $\mathcal{H}$), and $D f(\cdot)$ be
its given codifferential mapping. For any $\mu \ge 0$ denote
$$
  \overline{d}_{\mu} f(x) = \{ (b, w) \in \overline{d} f(x) \mid b \le \mu \}
$$
(cf.~\eqref{CodiffApprox_CorrectSigns}). Let us note that in the definition of $\overline{d}_{\mu} f(x)$ it is
sufficient to consider only extreme points $(b, w)$ of the hyperdifferential $\overline{d} f(x)$
(see~\cite{Dolgopolik_MCD}). A description of the original version of the method of codifferential descent (MCD)
\cite{DemRub_book} is given in Algorithm~\ref{algorithm_MCD}.

\begin{algorithm}[t]	\label{algorithm_MCD}
\caption{The method of codifferential descent (MCD).}

\noindent\textit{Step}~1. {Choose $\mu \ge 0$, a starting point $x_0 \in \mathcal{H}$, and set $n := 0$.}

\noindent\textit{Step}~2. {Compute $\underline{d} f(x_n)$ and $\overline{d}_{\mu} f(x_n)$.}

\noindent\textit{Step}~3. {For any $z = (b, w) \in \overline{d}_{\mu} f(x_n)$ compute 
$(a_n(z), v_n(z)) \in \mathbb{R} \times \mathcal{H}$ by solving
$$
  \min \| (a, v) \|^2 \quad \text{s.t. } (a, v) \in \underline{d} f(x_n) + z.
$$
}

\noindent\textit{Step}~4. {For any $z \in \overline{d}_{\mu} f(x_n)$ compute $\alpha_n(z) \ge 0$ by solving
$$
  \min_{\alpha} f(x_n - \alpha v_n(z)) \quad \text{s.t. } \alpha \ge 0.
$$
}

\noindent\textit{Step}~5. {Compute $z_n \in \overline{d}_{\mu} f(x_n)$ by solving
$$
  \min_z f(x_n - \alpha_n(z) v_n(z)) \quad \text{s.t. } z \in \overline{d}_{\mu} f(x_n).
$$
Set $x_{n + 1} = x_n - \alpha_n(z_n) v_n(z_n)$, $n := n + 1$, and go to Step 2.
}
\end{algorithm}

Note that in each iteration of the MCD one must perform line search in several directions 
(unless $\overline{d} f(\cdot) \equiv \{ 0 \}$; see Step~4). One can verify that at least one of these directions is a
descent direction of the function $f$, and $f(x_{n + 1}) < f(x_n)$ for all $n \in \mathbb{N} \cup \{ 0 \}$. On the other
hand, some of these directions might not be descent directions, i.e. the function $f$ may first increase and then
decrease in these directions. This interesting feature of the MCD allows it to ``jump over'' some points of local
minimum of the function $f$, provided the parameter $\mu > 0$ is sufficiently large (see \cite{DemBagRub} for a
particular example). However, no results on the convergence of the MCD to a global minimizer of the function $f$ are
known. The main goal of this article is to shed some light on this problem. To this end, below we study the performance
of the MCD in the case when $f$ is either convex or piecewise affine. For a comprehensive convergence analysis of the
MCD and its modifications in the general case see \cite{Dolgopolik_MCD}.

\section{The method of hypodifferential descent for convex optimization}
\label{Sect_ConvexCase}

In this section, we study the performance of the method of codifferential descent in the convex case. Let 
$f \colon \mathcal{H} \to \mathbb{R}$ be a convex function. As it was noted above, a function is codifferentiable
if and only if its increment can be locally approximated by the difference of convex function (i.e. a DC function).
If a codifferentiable function under consideration is convex, then it is natural to assume that its increment can be
approximated by a convex function. In other words, it is natural to suppose that $f$ is \textit{hypodifferentiable},
i.e. that there exists a codifferential mapping $D f(\cdot)$ such that $\overline{d} f(\cdot) \equiv \{ 0 \}$.
Furthermore, in this section we suppose that the function $f$ is continuously hypodifferentiable on $\mathcal{H}$, and
consider only its continuous hypo\-differential mapping $\underline{d} f(\cdot)$. Note that by \eqref{NessOptCond} 
a point $x^*$ is a global minimizer of $f$ if and only if $0 \in \underline{d} f(x^*)$, since in the convex case
$f'(x^*, \cdot) \ge 0$ if and only if $x^*$ is a global minimizer of $f$.

When the MCD is applied to a hypodifferentiable convex function, one calls it \textit{the method of
hypodifferential descent} (MHD). Moreover, in the convex case one can utilise Armijo's step-size rule
(cf.~\cite{Dolgopolik_MCD}). The scheme of the MHD for minimizing the function $f$ is given in
Algorithm~\ref{algorithm_MHD}.

\begin{algorithm}[t] \label{algorithm_MHD}
\caption{The method of hypodifferential descent (MHD).}

\noindent\textit{Step}~1. {Choose a starting point $x_0 \in \mathcal{H}$, $\sigma \in (0, 1)$ and $\gamma \in (0, 1)$,
and set $n := 0$.
}

\noindent\textit{Step}~2. {Compute $\underline{d}f(x_n)$.}

\noindent\textit{Step}~3. {Compute $(a_n, v_n) \in \mathbb{R} \times \mathcal{H}$ by solving
$$
  \min \| (a, v) \|^2 \quad \text{s.t. } (a, v) \in \underline{d} f(x_n).
$$
}

\noindent\textit{Step}~4. {Compute $k \in \mathbb{N} \cup \{ 0 \}$ by solving
$$
  \max_{k \in \mathbb{N} \cup \{ 0 \}} \gamma^k \quad
  \text{s.t. } f(x_n - \gamma^k v_n) - f(x_n) \le - \gamma^k \sigma \| (a_n, v_n) \|^2,
$$
and set $\alpha_n = \gamma^k$.
}

\noindent\textit{Step}~5. {Set $x_{n + 1} = x_n - \alpha_n v_n$, $n := n + 1$, and go to Step~2.
}
\end{algorithm}
Let us note that by \cite[Lemma~1]{Dolgopolik_MCD} one has $f'(x_n, -v_n) \le -\| (a_n, v_n) \|^2$. Hence by the
definition of directional derivative for any sufficiently small $\alpha > 0$ one has
\begin{equation} \label{RelaxationAlongMHD_Direction}
  f(x_n - \alpha v_n) - f(x_n) \le \alpha \sigma f'(x_n, -v_n) \le - \alpha \sigma \| (a_n, v_n) \|^2,
\end{equation}
if $\| (a_n, v_n) \| > 0$, i.e. $0 \notin \underline{d} f(x_n)$. Therefore, the step sizes $\alpha_n$ (see Step~4
of the MHD) are correctly defined, and $f(x_{n + 1}) < f(x_n)$ for all $n \in \mathbb{N} \cup \{ 0 \}$, provided $x_n$
is not a point of global minimum of the function $f$. 

Our aim is to estimate a rate of convergence of the MHD for the function $f$. This problem is very complicated in the
general case due to the nonuniqueness of hypodifferential mapping. A poor choice of a hypodifferential mapping
might significantly slow down the convergence of the method. To overcome this difficulty we must assume that the chosen
hypodifferential mapping somehow agrees with the convexity of the function $f$. The following definition provides a
precise and natural formulation of this assumption.

\begin{definition}
Let $C \subseteq \mathcal{H}$ be a nonempty convex set. A hypodifferential mapping $\underline{d} f(\cdot)$ of the
function $f$ is called \textit{amenable} on $C$, if for any $x \in C$ and $(a, v) \in \underline{d} f(x)$ one has
$$
  f(y) - f(x) \ge a + \langle v, y - x \rangle \quad \forall y \in C.
$$
\end{definition}

Clearly, if $f$ is continuously differentiable, then $\underline{d} f(\cdot) = \{ (0, \nabla f(\cdot)) \}$ is an
amenable continuous hypodifferential mapping of the function $f$ on any convex set $C$, since
$$
  f(y) - f(x) \ge \langle \nabla f(x), y - x \rangle \quad \forall x, y \in \mathcal{H}
$$
due to the convexity of the function $f$. Moreover, the amenability of hypodifferential mapping is preserved under
addition and pointwise maximum. 

\begin{proposition} \label{Prp_AmenabilitySumMax}
Let convex functions $f_i \colon \mathcal{H} \to \mathbb{R}$ be hypodifferentiable, and 
$\underline{d} f_i(\cdot)$ be their hypodifferential mappings that are amenable on a convex set 
$C \subseteq \mathcal{H}$, $i \in I = \{ 1, \ldots, k \}$. Then 
\begin{equation} \label{HypodiffOfLinComb}
  \underline{d} g(\cdot) = \sum_{i = 1}^k \lambda_i \underline{d} f_i(\cdot)
\end{equation}
is a hypodifferential mapping of the function $g = \sum_{i = 1}^k \lambda_i f_i$ that is amenable on $C$ 
(here $\lambda_i \ge 0$), and
\begin{equation} \label{HypodiffOfMax}
  \underline{d} u(\cdot) = \co\Big\{ (f_i(\cdot) - u(\cdot), 0) + \underline{d} f_i(\cdot) \Bigm| 1 \le i \le k \Big\}
\end{equation}
is a hypodifferential mapping of the function $u = \max_{i \in I} f_i$ that is amenable on $C$ as well.
\end{proposition}

\begin{proof}
Fix arbitrary $x, \Delta x \in \mathcal{H}$. By the definition of hypodifferentiable function for any $i \in I$ one has
$$
  f_i(x + \alpha \Delta x) - f_i(x) = 
  \max_{(a, v) \in \underline{d} f_i(x)} (a + \alpha \langle v, \Delta x \rangle )
  + o_i(\alpha),
$$
where $o_i(\alpha) / \alpha \to 0$ as $\alpha \to +0$. Hence
\begin{align} \notag
  g(x + \alpha \Delta x) - g(x) &= \sum_{i = 1}^k \lambda_i \big( f_i(x + \alpha \Delta x) - f_i(x) \big) \\
  &= \sum_{i = 1}^k \lambda_i \max_{(a, v) \in \underline{d} f_i(x)} (a + \alpha \langle v, \Delta x \rangle )
  + \sum_{i = 1}^k \lambda_i o_i(\alpha). \label{SumHypoDiffApproxEstimate}
\end{align}
Observe that
\begin{equation} \label{HypodiffOfSumRepresentation}
  \sum_{i = 1}^k \lambda_i \max_{(a, v) \in \underline{d} f_i(x)} (a + \alpha \langle v, \Delta x \rangle )
  = \max_{(a, v) \in \underline{d} g(x)} (a + \alpha \langle v, \Delta x \rangle),
\end{equation}
where $\underline{d} g(x)$ is defined as in \eqref{HypodiffOfLinComb}. Consequently, \eqref{SumHypoDiffApproxEstimate}
implies that
$$
  \Big| g(x + \alpha \Delta x) - g(x) - \max_{(a, v) \in \underline{d} g(x)} (a + \alpha \langle v, \Delta x \rangle)
  \Big| \le \sum_{i = 1}^k \lambda_i |o_i(\alpha)|.
$$
Therefore $g$ is hypodifferentiable, and \eqref{HypodiffOfLinComb} is its hypodifferential mapping. Let us check that
it is amenable on $C$. Indeed, fix $x, y \in C$ and $(a, v) \in \underline{d} g(x)$. By \eqref{HypodiffOfLinComb} there
exists $(a_i, v_i) \in \underline{d} f_i(x)$ such that 
\begin{equation} \label{HypogradOfConicCombination}
  (a, v) = \sum_{i = 1}^k \lambda_i (a_i, v_i).
\end{equation}
From the fact that the hypodifferentials $\underline{d} f_i(x)$ are amenable on $C$ it follows that
$$
  f_i(y) - f_i(x) \ge a_i + \langle v_i, y - x \rangle \quad \forall i \in I.
$$
Multiplying these inequalities by $\lambda_i$ and summing them up one obtains that
$$
  g(y) - g(x) = \sum_{i = 1}^k \big( \lambda_i f_i(y) - \lambda_i f_i(x) \big)
  \ge \sum_{i = 1}^k \lambda_i \big( a_i + \langle v_i, y - x \rangle \big) = a + \langle v, y - x \rangle,
$$
where the last equality follows from \eqref{HypogradOfConicCombination}. Thus, hypodifferential mapping
\eqref{HypodiffOfLinComb} of the function $g$ is amenable.

Let us now turn to the function $u$. By the definition of hypodifferentiable function one has
\begin{align*}
  u(x + \alpha \Delta x) - u(x) 
  &= \max_{i \in I} \big( f_i(x + \alpha \Delta x) - u(x) \big) \\
  &= \max_{i \in I} \Big( f_i(x) - u(x) + 
  \max_{(a, v) \in \underline{d} f_i(x)} (a + \alpha \langle v, \Delta x \rangle) + o_i(\alpha) \Big).
\end{align*}
Consequently, taking into account the fact that
\begin{equation} \label{HypodiffOfMaxRepresentation}
  \max_{(a, v) \in \underline{d} u(x)} (a + \alpha \langle v, \Delta x \rangle)
  = \max_{i \in I} \Big( f_i(x) - u(x) +
  \max_{(a, v) \in \underline{d} f_i(x)} (a + \alpha \langle v, \Delta x \rangle) \Big)
\end{equation}
(here $\underline{d} u(x)$ is defined as in \eqref{HypodiffOfMax}), and applying the inequality
\begin{equation} \label{MinMaxNumberInequal}
  \min_{i \in I} d_i \le \max_{i \in I} (c_i + d_i) - \max_{i \in I} c_i \le \max_{i \in I} d_i,
\end{equation}
which is valid for any $c_i, d_i \in \mathbb{R}$, with 
$c_i = f_i(x) - u(x) + \max_{(a, v) \in \underline{d} f_i(x)} (a + \alpha \langle v, \Delta x \rangle)$ and
$d_i = o_i(\alpha)$ one obtains that
$$
  \min_{i \in I} o_i(\alpha) \le
  u(x + \alpha \Delta x) - u(x) - \max_{(a, v) \in \underline{d} u(x)} (a + \alpha \langle v, \Delta x \rangle)
  \le \max_{i \in I} o_i(\alpha).
$$
Hence with the use of the inequality $\min_{i \in I} d_i \ge - \max_{i \in I} |d_i|$ one gets
\begin{equation} \label{MaxHypoDiffApproximEstimate}
  \Big| u(x + \alpha \Delta x) - u(x) - 
  \max_{(a, v) \in \underline{d} u(x)} (a + \alpha \langle v, \Delta x \rangle) \Big|
  \le \max_{i \in I} |o_i(\alpha)|,
\end{equation}
which implies that the function $u$ is hypodifferentiable, and \eqref{HypodiffOfMax} is its hypodifferential mapping.
Let us show that this mapping is amenable on $C$. Indeed, fix any $x, y \in C$ and $(a, v) \in \underline{d} u(x)$. 
By \eqref{HypodiffOfMax} there exist $\alpha_i \ge 0$ and $(a_i, v_i) \in \underline{d} f_i(x)$, $i \in I$ such that
$$
  (a, v) = \sum_{i = 1}^k \alpha_i (f_i(x) - u(x), 0) + \sum_{i = 1}^k \alpha_i (a_i, v_i),
  \qquad \sum_{i = 1}^k \alpha_i = 1.
$$
With the use of the amenability of $\underline{d} f_i(x)$ on $C$ one gets that
\begin{multline*}
  u(y) - u(x) = \max_{i \in I} \big( f_i(y) - u(x) \big) 
  \ge \max_{i \in I} \big( f_i(x) - u(x) + a_i + \langle v_i, y - x \rangle \big) \\
  \ge \sum_{i = 1}^k \alpha_i \big( f_i(x) - u(x) + a_i + \langle v_i, y - x \rangle \big)
  = a + \langle v, y - x \rangle
\end{multline*}
for any $y \in \mathcal{H}$, which implies the required result. 
\end{proof}

In the smooth case the rate of convergence of gradient methods for convex minimization is typically estimated
under the assumption that the gradient of the objective function is globally Lipschitz continuous (cf.~\cite{Nesterov}).
Therefore, it is natural to expect that in order to estimate the rate of convergence of the MHD in the nonsmooth case
we have to utilise a generalization of this assumption.

Recall that if a function $f$ is differentiable, and its gradient is globally Lipschitz continuous with Lipschitz
constant $L$, then
$$
  \big| f(y) - f(x) - \langle \nabla f(x), y - x \rangle \big| \le \frac{L}{2} \| y - x \|^2
  \quad \forall x, y \in \mathcal{H}
$$
(see, e.g., \cite[Lemma~1.2.3]{Nesterov}). We use this inequality as a basis for the generalization of the Lipschitz
continuity assumption to the nonsmooth case.

\begin{definition} \label{Def_HypodiffLipschitz}
Let $C \subseteq \mathcal{H}$ be a nonempty set. One says that a hypodifferential mapping $\underline{d} f(\cdot)$ is a
\textit{Lipschitzian approximation} of the function $f$ on the set $C$ with Lipschitz constant $L > 0$, if
$$
  \Big| f(y) - f(x) 
  - \max_{(a, v) \in \underline{d} f(x)} \big( a + \langle v, y - x \rangle \big) \Big|
  \le \frac{L}{2} \| y - x \|^2.
$$
for all $x, y \in C$.
\end{definition}

From the proof of Proposition~\ref{Prp_AmenabilitySumMax}, it follows that the property of being a
Lipschitzian approximation is preserved under addition and pointwise maximum (see \eqref{SumHypoDiffApproxEstimate} and
\eqref{MaxHypoDiffApproximEstimate}). Namely, the following result holds true.

\begin{proposition}
Let convex functions $f_i \colon \mathcal{H} \to \mathbb{R}$ be hypodifferentiable, and 
$\underline{d} f_i(\cdot)$ be their hypodifferential mappings, $i \in I = \{ 1, \ldots, k \}$. Suppose that for any 
$i \in I$ the mapping $\underline{d} f_i(\cdot)$ is a Lipschitzian approximation of the function $f_i$ on a set 
$C \subseteq \mathcal{H}$ with Lipschitz constant $L_i > 0$. Then the hypodifferential mapping
\eqref{HypodiffOfLinComb} is a Lipschitzian approximation of the function $g = \sum_{i = 1}^k \lambda_i f_i$ on the set
$C$ with Lipschitz constant $L \le \sum_{i = 1}^k | \lambda_i | L_i$ (here $\lambda_i \in \mathbb{R}$), and
\eqref{HypodiffOfMax} is a Lipschitzian approximation of the function $u = \max_{1 \le i \le k} f_i$ on the set $C$
with Lipschitz constant $L \le \max_{1 \le i \le k} L_i$.
\end{proposition}

\begin{proof}
Fix any $x, y \in C$. With the use of \eqref{HypodiffOfSumRepresentation} and Def.~\ref{Def_HypodiffLipschitz} one
obtains that
\begin{align*}
  \big| g(y) - g(x) &- \max_{(a, v) \in \underline{d} g(x)} (a + \langle v, y - x \rangle) \big| 
  \\
  &\le \sum_{i = 1}^k 
  |\lambda_i| \big| f_i(y) - f_i(x) - \max_{(a, v) \in \underline{d} f_i(x)} (a + \langle v, y - x \rangle \big|
  \\
  &\le \sum_{i = 1}^k |\lambda_i| \frac{L_i}{2} \| y - x \|^2.
\end{align*}
Therefore the hypodifferential mapping \eqref{HypodiffOfLinComb} is a Lipschitzian approximation of the function 
$g = \sum_{i = 1}^k \lambda_i f_i$ on the set $C$ with Lipschitz constant $L \le \sum_{i = 1}^k | \lambda_i | L_i$.

To prove the assertion for the function $u$, denote
\begin{equation} \label{HypodiffMaxIncrementApprox}
  \omega_i(y, x) = f_i(y) - f_i(x) - \max_{(a, v) \in \underline{d} f_i(x)} (a + \langle v, y - x \rangle).
\end{equation}
By definition one has
\begin{align*}
  u(y) - u(x) &= \max_{i \in I} (f_i(y) - u(x)) \\
  &= \max_{i \in I} \Big( f_i(x) - u(x) 
  + \max_{(a, v) \in \underline{d} f_i(x)} (a + \langle v, y - x \rangle) + \omega_i(y, x) \Big).
\end{align*}
Subtracting $\max_{(a, v) \in \underline{d} u(x)} (a + \langle v, y - x \rangle)$ 
(see~\eqref{HypodiffOfMaxRepresentation}), and applying inequality \eqref{MinMaxNumberInequal} with 
$c_i = f_i(x) - u(x) + \max_{(a, v) \in \underline{d} f_i(x)} (a + \langle v, y - x \rangle)$, and
$d_i = \omega_i(y, x)$ one obtains that
$$
  \min_{i \in I} \omega_i(y, x) \le u(y) - u(x) - \max_{(a, v) \in \underline{d} u(x)} (a + \langle v, y - x \rangle)
  \le \max_{i \in I} \omega_i(y, x).
$$
Therefore
$$
  \big| u(y) - u(x) - \max_{(a, v) \in \underline{d} u(x)} (a + \langle v, y - x \rangle) \big| 
  \le \max_{i \in I} |\omega_i(y, x)|.
$$
From \eqref{HypodiffMaxIncrementApprox} and the fact that $\underline{d} f_i(\cdot)$ is a Lipschitzian approximation of
the function $f_i$ on $C$ with Lipschitz constant $L_i$ it follows that
$$
  |\omega_i(y, x)| \le \frac{L_i}{2} \| y - x \|^2, \quad
  \max_{i \in I} |\omega_i(y, x)| \le \frac{\max_{i \in I} L_i}{2} \| y - x \|^2,
$$
which implies the required result.
\end{proof}

Now, we can obtain an upper estimate of the rate of convergence of the MHD that coincides with the upper estimate of
the rate of convergence of the standard gradient method in the convex case (see, e.g., \cite[Theorem~2.1.14]{Nesterov}).
This result is not surprising since in the smooth case the MHD is reduced to the gradient method with Armijo's
step-size rule. Let us note that the proof of the following theorem is a straightforward modification of the proof of
the corresponding result for gradient methods to the nonsmooth case.

\begin{theorem} \label{Theorem_MHD_RateOfConvergence}
Let $f$ be a closed convex function, the set 
\[
  S_0 = \{ x \in \mathcal{H} \mid f(x) \le f(x_0) \}
\]
be bounded, and let the continuous hypodifferential mapping $\underline{d} f(\cdot)$ be amenable and bounded on the set
$S_0$. Suppose also that $\underline{d} f(\cdot)$ is a Lipschitzian approximation of $f$ on the set 	
$S_\varepsilon = \{ x \in \mathcal{H} \mid \dist(x, S_0) \le \varepsilon \}$ for some $\varepsilon > 0$, and a sequence
$\{ x_n \}$ is generated by the MHD. Then there exists $\widehat{\alpha} > 0$ such that $\alpha_n \ge \widehat{\alpha}$
for all $n \in \mathbb{N}$, and the following inequality holds true:
\begin{equation} \label{ConvexCase_RateOfConvergence}
  f(x_n) - f(x^*) \le \frac{(f(x_0) - f(x^*)) R^2}{R^2 + (f(x_0) - f(x^*))\widehat{\alpha} \sigma n} 
  = \mathcal{O}\left( \frac{1}{n} \right) \quad \forall n \in \mathbb{N},
\end{equation}
Here $x^*$ is a point of global minimum of $f$, and $R = 1 + \sup_{n \ge 0} \| x_n - x^* \| < + \infty$.
\end{theorem}

\begin{proof}
At first, let us note that $f$ attains a global minimum by \cite[Prop.~II.1.2]{EkelandTemam}, since $f$ is closed, 
the set $S_0$ is bounded, and $\mathcal{H}$ is a Hilbert space. Note also that 
$R = 1 + \sup_{n \ge 0} \| x_n - x^* \|$ is finite due to the facts that $\{ x_n \} \subset S_0$, and $S_0$ is bounded
(the validity of the inclusion follows from the inequality $f(x_{n + 1}) < f(x_n)$;
see~\eqref{RelaxationAlongMHD_Direction}).

Denote $\Phi_n(y) = \max_{(a, v) \in \underline{d} f(x_n)} (a + \langle v, y \rangle)$. Applying the necessary and
sufficient condition for a minimum of a convex function on a convex set \cite[Proposition~II.2.1]{EkelandTemam} one
obtains that
\begin{equation} \label{SeparationThrm_MHD}
  a_n a + \langle v_n, v \rangle \ge \| (a_n, v_n) \|^2 \quad \forall (a, v) \in \underline{d} f(x_n),
\end{equation}
where the pair $(a_n, v_n)$ is computed on Step~3 of the MHD. If $a_n = 0$, then taking into account the fact that 
$a \le 0$ for all $(a, v) \in \underline{d} f(x_n)$ (see~\eqref{CodiffApprox_CorrectSigns}) one gets that 
$$
  \Phi_n(- v_n) \le \max_{(a, v) \in \underline{d} f(x_n)} \langle v, - v_n \rangle \le - \| (a_n, v_n) \|^2,
$$
which with the use of the convexity of $\Phi_n$ and the equality $\Phi_n(0) = 0$
(see~\eqref{CodiffApprox_CorrectSigns}) implies that
\begin{equation} \label{HypodiffDecay_MHD_Simple}
  \Phi_n( - \alpha v_n) \le \alpha \Phi_n( - v_n) + (1 - \alpha) \Phi_n(0)
  \le - \alpha \| (a_n, v_n) \|^2 \quad \forall \alpha \in [0, 1].
\end{equation}
On the other hand, if $a_n < 0$, then dividing \eqref{SeparationThrm_MHD} by $a_n$, and taking the maximum over all 
$(a, v) \in \underline{d} f(x_n)$ one obtains
$$
  \Phi_n\left( \frac{1}{a_n} v_n \right) \le - \frac{1}{|a_n|} \| (a_n, v_n) \|^2.
$$
Applying the convexity of $\Phi_n$ and the equality $\Phi_n(0) = 0$ again one obtains that
$$
  \Phi_n\left( \frac{\alpha}{a_n} v_n \right) \le \alpha \Phi_n\left( \frac{1}{a_n} v_n \right) 
  \le - \frac{\alpha}{|a_n|} \| (a_n, v_n) \|^2 \quad \forall \alpha \in [0, 1].
$$
Combining this inequality with \eqref{HypodiffDecay_MHD_Simple} one gets that in either case
\begin{equation} \label{HypodiffDecay_MHD}
  \Phi_n( - \alpha v_n) \le - \alpha \| (a_n, v_n) \|^2 
  \quad \forall \alpha \in \left[ 0, \frac{1}{|a_n|} \right],
\end{equation}
where $1 / 0 = 1$ by definition. Observe that the sequence $\{ a_n \}$ is bounded by virtue of the facts that 
$\{ x_n \} \subset S_0$, and the hypodifferential mapping $\underline{d} f(\cdot)$ is bounded on $S_0$. Therefore, there
exists $\varkappa \in (0, 1]$ such that $|a_n|^{-1} > \varkappa > 0$ for all $n \in \mathbb{N}$. Furthermore, from the
boundedness of $\underline{d} f(\cdot)$ on $S_0$ it follows that there exists $K > 0$ such that $\| v_n \| \le K$ for
all $n \in \mathbb{N}$. Hence, in particular, 
$x_n - \alpha v_n \in S_{\varepsilon} = \{ x \in \mathcal{H} \mid \dist(x, S_0) \le \varepsilon \}$ for any 
$\alpha \in [0, \varepsilon / K]$ and $n \in \mathbb{N}$.

Recall that $\underline{d} f(\cdot)$ is a Lipschitzian approximation of $f$ on $S_\varepsilon$. Therefore there exists
$L > 0$ such that
$$
  f(x_n - \alpha v_n) - f(x_n) - \Phi_n(-\alpha v_n) \le \frac{L \alpha^2}{2} \| v_n \|^2
  \quad \forall \alpha \in \left[ 0, \frac{\varepsilon}{K} \right].
$$
Hence and from \eqref{HypodiffDecay_MHD} it follows that
$$
  f(x_n - \alpha v_n) - f(x_n) \le \left( - \alpha + \frac{L \alpha^2}{2} \right) \| (a_n, v_n) \|^2 
  \quad \forall \alpha \in \left[ 0, \min\left\{ \varkappa, \frac{\varepsilon}{K} \right\} \right].
$$
Consequently, as it is easy to see, there exists $\widehat{\alpha} > 0$ such that
$$
  f(x_n - \widehat{\alpha} v_n) - f(x_n) \le - \widehat{\alpha} \sigma \| (a_n, v_n) \|^2
  \quad \forall n \in \mathbb{N}
$$
(one can choose any $\widehat{\alpha} \le \min\{ 2 ( 1 - \sigma) / L, \varkappa, \varepsilon / K \}$), which implies
that 
\begin{equation} \label{MHD_SufficientDescentInequal}
  f(x_{n + 1}) - f(x_n) \le - \widehat{\alpha} \sigma \| (a_n, v_n) \|^2, 
  \quad \alpha_n \ge \widehat{\alpha} \quad \forall n \in \mathbb{N},
\end{equation}
where $x_{n + 1} = x_n - \alpha_n v_n$, and $\alpha_n$ is computed on Step~4 of the MHD. Note that one can 
set $\widehat{\alpha} = \gamma^k$ for a sufficiently large $k \in \mathbb{N}$. Then $\alpha_n = \gamma^{k_n}$ with
$k_n \le k$.

Denote $\Delta_n = f(x_n) - f(x^*)$, where $x^*$ is a point of global minimum of the function $f$. From the facts that
the hypodifferential mapping $\underline{d} f(\cdot)$ is amenable, and $(a_n, v_n) \in \underline{d} f(x_n)$ (see
Step~3 of the MHD) it follows that
$$
  \Delta_n \le - a_n + \langle v_n, x_n - x^* \rangle \le \| (a_n, v_n) \| \big( 1 + \| x_n - x^* \| \big)
  \le R \| (a_n, v_n) \|
$$
(recall that $R = 1 + \sup_{n \ge 0} \| x_n - x^* \|$). Adding and subtracting $f(x^*)$ in
\eqref{MHD_SufficientDescentInequal}, and estimating $\| (a_n, v_n) \|^2$ with the use of the inequality above one gets
that
$$
  \Delta_{n + 1} \le \Delta_n - \frac{\widehat{\alpha} \sigma}{R^2} \Delta_n^2.
$$
Dividing this inequality by $\Delta_n \cdot \Delta_{n + 1}$ one obtains
$$
  \frac{1}{\Delta_{n + 1}} \ge \frac{1}{\Delta_n} + \frac{\widehat{\alpha} \sigma}{R^2} \frac{\Delta_n}{\Delta_{n + 1}}
  \ge \frac{1}{\Delta_n} + \frac{\widehat{\alpha} \sigma}{R^2}
$$
(note that $\Delta_{n + 1} \le \Delta_n$ due to the fact that $f(x_{n + 1}) \le f(x_n)$). Summing up these inequalities
one gets
$$
  \frac{1}{\Delta_{n + 1}} \ge \frac{1}{\Delta_0} + \frac{\widehat{\alpha} \sigma}{R^2} (n + 1) 
  \quad \forall n \in \mathbb{N},
$$
which implies that \eqref{ConvexCase_RateOfConvergence} is valid.
\end{proof}

\begin{remark}
Let us point out how $\widehat{\alpha}$ from the theorem above depends on the problem data. Let $K > 0$ be such that
$|a| \le K$ and $\| v \| \le K$ for all $(a, v) \in \underline{d} f(x)$ and $x \in S_0$. Then, as it was pointed out in
the proof, one can set
$$
  \widehat{\alpha} = \min\left\{ \frac{\min\{ 1, \varepsilon \}}{K}, \frac{2 (1 - \sigma)}{L} \right\}.
$$
Furthermore, if $\varepsilon = + \infty$, then it is sufficient to suppose that $K > 0$ is such that $|a| \le K$ for
any $(a, v) \in \underline{d} f(x)$ and $x \in S_0$. Note that in the smooth case one can define 
$\underline{d} f(\cdot) = \{ (0, \nabla f(\cdot)) \}$, which implies that $\widehat{\alpha} = 2(1 - \sigma) / L$,
provided the gradient $\nabla f(\cdot)$ is globally Lipschitz continuous. Observe also that the theorem above remains
valid in the case when instead of Armijo's step-size rule one finds $\alpha_n$ via the minimization of the function
$\alpha \mapsto f(x_n - \alpha v_n)$.
\end{remark}

\begin{remark}
Note that the rate of convergence of the MHD is better than the optimal rate of convergence of subgradient methods
$\mathcal{O}(1 / \sqrt{n})$ \cite[Sect.~3.2]{Nesterov}. This is obviously due to the fact the oracle utilised by the
MHD provides much more information about the objective function than just a single subgradient. On the other hand, each
call of this oracle is significantly more expensive than the call of the oracle used in subgradient methods. Let us
also note that one can utilise Nesterov's acceleration technique \cite[Sect.~2.2]{Nesterov} to design a faster method
for minimizing hypodifferentiable convex functions than the MHD. However, this method must accumulate the Minkowski sum
of the form $a_1 \underline{d} f(y_1) + a_2 \underline{d} f(y_2) + \ldots$ with some $a_i \in \mathbb{R}$ (cf. 
the optimal gradient method in \cite{Nesterov}), which is unreasonable both in terms of memory consumption and
computational effort. That is why we do not present an accelerated version of the MHD here.
\end{remark}

\begin{remark}
Let $U \subset \mathcal{H}$ be a bounded open set such that $f$ is Lipschitz continuous on $U$. 
By \cite[Example~4]{Dolgopolik_MCD} for any $x \in U$ one has $f(x) = \max_{(a, v) \in C} (a + \langle v, x \rangle)$,
where
\begin{equation} \label{AffineSupportOfConvexFunction}
  C = \big\{ (f(z) - \langle v, z \rangle, v) \in \mathbb{R} \times \mathcal{H} \bigm| 
  v \in \partial f(z), \: z \in U \big\},
\end{equation}
and $\partial f(z)$ is the subdifferential of $f$ at $z$ in the sense of convex analysis. Therefore, 
for any $x, y \in U$ one has
$$
  f(y) - f(x) = \max_{(a, v) \in C} (a + \langle v, y \rangle) - f(x) 
  = \max_{(a, v) \in C} (a - f(x) + \langle v, x \rangle + \langle v, y - x \rangle)
$$
or, equivalently, $f(y) - f(x) = \max_{(a, v) \in \underline{d} f(x)} (a + \langle v, y - x \rangle)$, where
\begin{equation} \label{Hypodiff_via_Subdiff}
  \underline{d} f(x) = \cl\co\Big\{ 
  \big( f(z) - f(x) - \langle v, z - x \rangle, v \big) \in \mathbb{R} \times \mathcal{H} \Bigm|
  v \in \partial f(z), \: z \in U \Big\}
\end{equation}
(see~\eqref{AffineSupportOfConvexFunction}). Applying the fact that $f$ is Lipschitz continuous on $U$ one can verify
that the multifunction $\underline{d} f(\cdot)$ is Hausdorff continuous and bounded on $U$. Note that this
hypodifferential mapping is obviously a Lipschitzian approximation of $f$ on $U$. Furthermore, observe that from the
inequality $f(y) - f(z) \ge \langle v, y - z \rangle$, where $y, z \in U$ and $v \in \partial f(z)$, it follows that 
$$
  f(y) - f(x) \ge f(z) - f(x) - \langle v, z - x \rangle + \langle v, y - x \rangle
  \quad \forall x, y, z \in U \: \forall v \in \partial f(z).
$$
With the use of this inequality and \eqref{Hypodiff_via_Subdiff} one can check that the hypodifferential mapping
\eqref{Hypodiff_via_Subdiff} is amenable on $U$. Thus, if the sublevel set 
$S_0 = \{ x \in \mathcal{H} \mid f(x) \le f(x_0) \}$ is bounded, and the exists $\varepsilon > 0$ such that $f$ is
Lipschitz continuous on $S_{\varepsilon}$, then there exists a hypodifferential mapping of $f$ (of the form
\eqref{Hypodiff_via_Subdiff}) satisfying the assumptions of Theorem~\ref{Theorem_MHD_RateOfConvergence}. In particular,
if $\mathcal{H}$ is finite dimensional, then the boundedness of the sublevel set $S_0$ guarantees that there exists a
hypodifferential mapping of the function $f$ satisfying the assumptions of Theorem~\ref{Theorem_MHD_RateOfConvergence}.
Thus, at least from the theoretical point of view the assumptions of this theorem are not very restrictive.
\end{remark}

\section{Codifferential calculus and global piecewise af\-fine optimization}
\label{Sect_PiecewiseAffineCase}

The main goal of this section is to demonstrate that the method of codifferential descent finds a point of
\textit{global} minimum of a nonconvex piecewise affine function in a finite number of steps. To this end, we derive new
necessary and sufficient conditions for a global minimum of a piecewise affine functions in terms of its
codifferential, which significantly differ from the ones obtained in \cite{Polyakova}, and develop a modification of
the MCD call \textit{the method of global codifferential descent}.

\subsection{Global codifferential and optimality conditions}

From this point onwards we suppose that $\mathcal{H} = \mathbb{R}^d$, and write $\mathbb{R}^{d + 1}$ instead of
$\mathbb{R} \times \mathbb{R}^d$. We start with an auxiliary result for polyhedral convex functions.

\begin{lemma} \label{Lemma_PolyhedralFunc_Nonnegative}
Let a function $f \colon \mathbb{R}^d \to \mathbb{R}$ have the form 
$f(x) = \max_{i \in I} (a_i + \langle v_i, x \rangle)$ for some $(a_i, v_i) \in \mathbb{R}^{d + 1}$,
where $I = \{ 1, \ldots, k \}$. Then $f(x) \ge 0$ for all $x \in \mathbb{R}^d$ if and only if either 
$0 \in C = \co\{ (a_i, v_i) \mid i \in I \}$ or $f$ is bounded below and $a_0 > 0$, where
\begin{equation} \label{PolyConvex_MinHypoGrad}
  \{ (a_0, v_0) \} = \argmin\big\{ \| (a, v) \|^2 \mid (a, v) \in C \big\}.
\end{equation}
\end{lemma}

\begin{proof}
Let $f(x) \ge 0$ for all $x \in \mathbb{R}^d$. Arguing by reductio ad absurdum, suppose that $0 \notin C$, but 
$a_0 \le 0$. Applying the necessary and sufficient condition for a minimum of a convex function on a convex set
\cite[Proposition~II.2.1]{EkelandTemam} one obtains that
\begin{equation} \label{PolyhedralFunc_SeparationThrm}
  a_0 (a - a_0) + \langle v_0, v - v_0 \rangle \ge 0 \quad \forall (a, v) \in C.
\end{equation}
If $a_0 = 0$, then $v_0 \ne 0$ (otherwise $0 \in C$), and $\langle v, - v_0 \rangle \le - \| v_0 \|^2$ for any 
$(a, v) \in C$. Therefore, for all $\alpha \ge 0$ one has
\begin{equation} \label{PolyhedralFunc_UnboundedBelow}
  f(- \alpha v_0) = \max_{i \in I} (a_i + \langle v_i, - \alpha v_0 \rangle) \le
  \max_{i \in I} (a_i - \alpha \| v_0 \|^2) = f(0) - \alpha \| v_0 \|^2,
\end{equation}
which contradicts the assumption that $f$ is nonnegative.

If $a_0 < 0$, then dividing \eqref{PolyhedralFunc_SeparationThrm} by $a_0$ one obtains that
$$
  a + \left\langle v, \frac{1}{a_0} v_0 \right\rangle \le - \frac{1}{|a_0|} \| (a_0, v_0) \|^2 < 0 
  \quad \forall (a, v) \in C.
$$
Taking the maximum over all $(a, v) \in C$ one gets that $f(a_0^{-1} v_0) < 0$, which is impossible. Thus, $a_0 > 0$.

Let us prove the converse statement. If $0 \in C$, then for any $x \in \mathbb{R}^d$ one has
$f(x) = \max_{(a, v) \in C} (a + \langle v, x \rangle) \ge 0 + \langle 0, x \rangle = 0$, i.e. the function $f$ is
nonnegative. Arguing by reductio ad absurdum suppose now that $0 \notin C$, $f$ is bounded below, and $a_0 > 0$, but
there exists $x \in \mathbb{R}^d$ such that $f(x) < 0$.

Define $f^* = \inf_{x \in \mathbb{R}^d} f(x)$. By our assumptions $- \infty < f^* < 0$. Our aim is to show that
$(f^*, 0) \in C$. Then for any $\alpha \in [0, 1]$ one has $(1 - \alpha) (f^*, 0) + \alpha (a_0, v_0) \in C$. Setting
$\alpha = |f^*| / (|f^*| + a_0) \in (0, 1)$ one gets that $(0, \alpha v_0) \in C$, which is impossible due to the
definition of $(a_0, v_0)$ (see~\eqref{PolyConvex_MinHypoGrad}), and the fact 
that $\| (0, \alpha v_0) \|^2 < \| (a_0, v_0) \|^2$. 

For any $\varepsilon > 0$ there exists $x_{\varepsilon}$ such that $f(x_{\varepsilon}) < f^* + \varepsilon$. Hence by
definition $0 \in \partial_{\varepsilon} f(x_{\varepsilon})$, where $\partial_{\varepsilon} f(x_{\varepsilon})$ is 
the $\varepsilon$-subdifferential of $f$ at $x_{\varepsilon}$. By \cite[Example~XI.3.5.3]{HiriartUrruty} one has
$$
  \partial_{\varepsilon} f(x) =
  \big\{ v \in \mathbb{R}^d \bigm| \exists (a, v) \in C \colon a + \langle v, x \rangle \ge f(x) - \varepsilon \big\}.
$$
Consequently, for any $\varepsilon > 0$ there exists 
$a_{\varepsilon} \ge f(x_{\varepsilon}) - \varepsilon \ge f^* - \varepsilon$ such that
$(a_{\varepsilon}, 0) \in C$. Observe that for any $(a, 0) \in C$ one has $f(x) \ge a$ for all $x \in \mathbb{R}^d$,
which implies that $f^* \ge a$. Thus, $f^* \ge a_{\varepsilon} \ge f^* - \varepsilon$. Hence passing to the limit
as $\varepsilon \to 0$, and taking into account the fact that the set $C$ is closed one obtains that $(f^*, 0) \in C$.
\end{proof}

\begin{corollary} \label{Corollary_PolyhedralFunc_Nonnegative}
Let all assumptions of Lemma~\ref{Lemma_PolyhedralFunc_Nonnegative} be valid, and suppose that $f$ is bounded below.
Then $f(x) \ge 0$ for all $x \in \mathbb{R}^d$ if and only if $a_0 \ge 0$. 
\end{corollary}

\begin{proof}
If $f$ is nonnegative, then by Lemma~\ref{Lemma_PolyhedralFunc_Nonnegative} either $a_0 > 0$ or $0 \in C$. In the
latter case, by definition one has $(a_0, v_0) = (0, 0)$, i.e. $a_0 = 0$.

Suppose now that $a_0 \ge 0$. If $a_0 > 0$, then $f$ is nonnegative by Lemma~\ref{Lemma_PolyhedralFunc_Nonnegative}.
Therefore, suppose that $a_0 = 0$. If $v_0 = 0$, then $0 \in C$ and, once again, $f$ is nonnegative by
Lemma~\ref{Lemma_PolyhedralFunc_Nonnegative}. On the other hand, if $v_0 \ne 0$, then, as it was shown in the proof of
the lemma (see \eqref{PolyhedralFunc_UnboundedBelow}), $f$ is unbounded below, which contradicts our assumptions.
\end{proof}

\begin{remark}
Let us note that the assumption on the boundedness below of the function $f$ cannot be discarded from
Lemma~\ref{Lemma_PolyhedralFunc_Nonnegative}. A simple counterexample is 
the function $f(x) = a + \langle v, x \rangle$ with $a > 0$ and $v \ne 0$.
\end{remark}

Now we turn to the study of piecewise affine functions. At first, let us recall the definition of piecewise affine
function \cite{Kripfgang,GorokhovikZorko}. A convex set $Q \subset \mathbb{R}^d$ is referred to as \textit{polyhedral},
if it can be represented as the intersection of a finite family of closed halfspaces. A finite family 
$\sigma = \{ Q_1, \ldots, Q_k \}$, $k \in \mathbb{N}$, of polyhedral sets is said to be a \textit{polyhedral partition}
of $\mathbb{R}^d$, if $\mathbb{R}^d = \cup_{i = 1}^k Q_i$, $\interior Q_i \ne \emptyset$ for $1 \le i \le k$, and the
interiors of the sets $Q_i$ are mutually disjoint. Finally, a function $f \colon \mathbb{R}^d \to \mathbb{R}$ is called
\textit{piecewise affine}, if there exists a polyhedral partition $\sigma = \{ Q_1, \ldots, Q_k \}$ of $\mathbb{R}^d$
such that the restriction of $f$ to each $Q_i$ is an affine function.

Let $f \colon \mathbb{R}^d \to \mathbb{R}$ be a piecewise affine function. Then by
\cite[Theorem~3.1]{GorokhovikZorko}, there exist $(a_i, v_i) \in \mathbb{R}^{d + 1}$, $i \in I = \{ 1, \ldots, l \}$,
and $(b_j, w_j) \in \mathbb{R}^{d + 1}$, $j \in J = \{ 1, \ldots, s \}$, such that
\begin{equation} \label{DCdecomp_PiecewiseAffFunc}
  f(x) = \max_{i \in I} (a_i + \langle v_i, x \rangle) + \min_{j \in J} (b_j + \langle w_j, x \rangle)
  \quad \forall x \in \mathbb{R}^d.
\end{equation}
Define 
\begin{equation} \label{ConvexConcaveParts_PiecewiseAffine}
  \underline{f}(x) = \max_{i \in I} (a_i + \langle v_i, x \rangle), \quad 
  \overline{f}(x) = \min_{j \in J} (b_j + \langle w_j, x \rangle).
\end{equation}
Then $f = \underline{f} - (- \overline{f})$ is a DC decomposition of the function $f$ 
(i.e. $f = \underline{f} - (- \overline{f})$ is a representation of the function $f$ as the difference of
convex functions). Introduce the set-valued mappings
\begin{equation} \label{DefOfGlobalCodiff_PiecewiseAff}
\begin{split}
  \underline{d} f(x) &= \co\big\{ (a_i - \underline{f}(x) + \langle v_i, x \rangle, v_i) \in \mathbb{R}^{d + 1} 
  \bigm| i \in I \big\}, \\
  \overline{d} f(x) &= \co\big\{ (b_j - \overline{f}(x) + \langle w_j, x \rangle, w_j) \in \mathbb{R}^{d + 1}
  \bigm| j \in J \big\}.
\end{split}
\end{equation}
Then, as it is easy to see, for any $x, \Delta x \in \mathbb{R}^{d + 1}$ one has
\begin{align} 
  f(x &+ \Delta x) - f(x) = \underline{f}(x + \Delta x) + \overline{f}(x + \Delta x) 
  - (\underline{f}(x) + \overline{f}(x)) \notag \\
  &= \max_{i \in I} (a_i + \langle v_i, x + \Delta x \rangle) +
  \min_{j \in J} (b_j + \langle w_j, x + \Delta x \rangle) 
  - (\underline{f}(x) + \overline{f}(x)) \notag \\
  &= \max_{i \in I} (a_i - \underline{f}(x) + \langle v_i, x \rangle + \langle v_i, \Delta x \rangle) \notag \\
  &+ \min_{j \in J} (b_j - \overline{f}(x) + \langle w_j, x \rangle + \langle w_j, \Delta x \rangle) \notag \\
  &= \max_{(a, v) \in \underline{d} f(x)} (a + \langle v, \Delta x \rangle)
  + \min_{(b, w) \in \overline{d} f(x)} (b + \langle w, \Delta x \rangle). \label{PiecewiseAffFunc_CodiffDecomp}
\end{align}
Furthermore, for any $x \in \mathbb{R}^d$ one has
\begin{equation} \label{HypodiffNonnegative}
  \max_{(a, v) \in \underline{d} f(x)} a = \max_{i \in I} (a_i + \langle v_i, x \rangle) - \underline{f}(x) 
  = \underline{f}(x) - \underline{f}(x) = 0,
\end{equation}
and, similarly, $\min_{(b, w) \in \overline{d} f(x)} b = 0$. Thus, the pair 
$D f(x) = [\underline{d} f(x), \overline{d} f(x)]$ is a codifferential of $f$ at $x$. On the other hand, codifferential
is defined as a local approximation of a nonsmooth function, while equality \eqref{PiecewiseAffFunc_CodiffDecomp} holds
true for all $x, \Delta x \in \mathbb{R}^{d + 1}$, i.e. globally.

\begin{definition}
The pair $D f = [\underline{d} f, \overline{d} f]$ defined by \eqref{DefOfGlobalCodiff_PiecewiseAff} is called
\textit{a global codifferential mapping} (or simply \textit{global codifferential}) of the function $f$ (associated with
the DC decomposition $f = \underline{f} - (-\overline{f})$). The multifunction $\underline{d} f$ is called \textit{a
global hypodifferential} of $f$, while the multifunction $\overline{d} f$ is called \textit{a global hyperdifferential} 
of $f$. 
\end{definition}

Note that a global codifferential mapping of a piecewise affine function is not unique, since there exists infinitely
many DC decompositions of a piecewise affine function of the form \eqref{DCdecomp_PiecewiseAffFunc}. Let us also point
out that a global codifferential mapping of a piecewise affine function was first implicitly utilised by 
Polyakova in \cite{Polyakova}.

With the use of the codifferential calculus \cite{DemRub_book,Dolgopolik_Gen,Dolgopolik_MCD} one can obtain some simple
calculus rules for global codifferentials of piecewise affine functions.

\begin{proposition} \label{Prp_GlobalCodiffCalc}
Let $f_m \colon \mathbb{R}^d \to \mathbb{R}$, $m \in M = \{ 1, \ldots, p \}$, be piecewise affine functions of the form
$f_m = \underline{f}_m + \overline{f}_m$, where
$$
  \underline{f}_m(x) = \max_{i \in I_m} (a_{mi} + \langle v_{mi}, x \rangle), \quad
  \overline{f}_m(x) =  \min_{j \in J_m} (b_{mj} + \langle w_{mj}, x \rangle),
$$
and $I_m = \{ 1, \ldots, l_m \}$, $J_m = \{ 1, \ldots, s_m \}$. Let also $D f_m$ be the global codifferential
mapping of the function $f_m$ associated with the DC decomposition $f_m = \underline{f}_m - (\overline{f}_m)$, 
$m \in M$, and let $f \colon \mathbb{R}^d \to \mathbb{R}$ be a given function. Then the following statements hold true:
\begin{enumerate}
\item{if $f(x) = a + \langle v, x \rangle$, then both $D f(\cdot) \equiv [ \{ (0, v) \}, \{ (0, 0) \} ]$ and
$D f(\cdot) \equiv [ \{ (0, 0) \}, \{ (0, v) \} ]$ are global codifferential mappings of the function $f$;
\label{Calc_AffineFunc}}

\item{if $f = f_1 + c$ for some $c \in \mathbb{R}$, then $D f = D f_1$;
\label{Calc_PlusConstant}}

\item{if $f = \lambda f_1$, then $D f = [ \lambda \underline{d} f_1, \lambda \overline{d} f_1 ]$ 
in the case $\lambda \ge 0$, and $D f = [ \lambda \overline{d} f_1, \lambda \underline{d} f_1 ]$ 
in the case $\lambda < 0$;
\label{Calc_ScalarMultiplication}}

\item{if $f = \sum_{m = 1}^p f_m$, then $D f = [ \sum_{m = 1}^p \underline{d} f_m, \sum_{m = 1}^p \overline{d} f_m ]$;
\label{Calc_Sum}}

\item{if $f = \max_{m \in M} f_m$, then 
$$
  D f(\cdot) = \bigg[ \co\bigg\{ (f_m(\cdot) - f(\cdot), 0) + \underline{d} f_m(\cdot)
  - \sum_{k \ne m} \overline{d} f_k(\cdot) \biggm| m \in M \bigg\}, \sum_{m = 1}^p \overline{d} f_m(\cdot)\bigg]
$$
is a global codifferential mapping of $f$;
\label{Calc_Max}}

\item{if $f = \min_{m \in M} f_m$, then 
$$
  D f(\cdot) = \bigg[ \sum_{m = 1}^p \underline{d} f_m(\cdot), 
  \co\bigg\{ (f_m(\cdot) - f(\cdot), 0) + \overline{d} f_m(\cdot) 
  - \sum_{k \ne m} \underline{d} f_k(\cdot) \biggm| m \in M  \bigg\} \bigg]
$$
is a global codifferential mapping of $f$.
\label{Calc_Min}}
\end{enumerate}
\end{proposition}

\begin{proof}
\noindent\ref{Calc_AffineFunc}. Define
$$
  \underline{f}'(x) = a + \langle v, x \rangle, \; \overline{f}'(x) = 0, \quad
  \underline{f}''(x) = 0, \: \overline{f}''(x) = a + \langle v, x \rangle.
$$
Then $f = \underline{f}' - (- \overline{f}')$ and $f = \underline{f}'' - ( - \overline{f}'')$ are two DC decompositions
of the function $f$. Applying the definition of global codifferential \eqref{DefOfGlobalCodiff_PiecewiseAff} one gets
that $D f(\cdot) = [ \{ (0, v) \}, \{ (0, 0) \} ]$ is a global codifferential of $f$ associated with the first DC
decomposition, while $D f(\cdot) = [ \{ (0, 0) \}, \{ (0, v) \} ]$ is a global codifferential of $f$ associated with
the second DC decomposition.

\noindent\ref{Calc_PlusConstant}. Define
$$
  \underline{f}(x) = \underline{f}_1(x) + c = \max_{i \in I_1} ( a_{1i} + c + \langle v_{1i}, x \rangle),
  \quad \overline{f}(x) = \overline{f}_1(x).
$$
Then $f = \underline{f} - (- \overline{f}(x))$ is a DC decomposition of the function $f$. Applying the definition
of global codifferential \eqref{DefOfGlobalCodiff_PiecewiseAff}, and the fact that
$$
  (a_{1i} + c - \underline{f}(x) + \langle v_{1i}, x \rangle, v_{1i})
  = (a_{1i} - \underline{f}_1(x) + \langle v_{1i}, x \rangle, v_{1i}),
$$
one gets that $\underline{d} f(x) = \underline{d} f_1(x)$ and $\overline{d} f(x) = \overline{d} f_1(x)$, i.e. 
$D f = D f_1$ is a global codifferential of $f$ associated with the DC decomposition 
$f = \underline{f} - (- \overline{f})$ defined above.

\noindent\ref{Calc_ScalarMultiplication}. Let $\lambda \ge 0$. Define $\underline{f}(x) = \lambda \underline{f}_1(x)$
and $\overline{f}(x) = \lambda \overline{f}_1(x)$. Then $f = \lambda f_1 = \underline{f} - (- \overline{f})$ is a DC
decomposition of the function $f$. By definition
$$
  \underline{f}(x) = \max_{i \in I_1} (\lambda a_{1i} + \langle \lambda v_{1i}, x \rangle), \quad
  \overline{f}(x) = \min_{j \in J_1} (\lambda b_{1i} + \langle \lambda w_{1i}, x \rangle).
$$
Hence with the use of \eqref{DefOfGlobalCodiff_PiecewiseAff} and the fact that
\begin{equation} \label{HypogradScalarMultiplication}
  (\lambda a_{1i} - \underline{f}(x) + \langle \lambda v_{1i}, x \rangle, \lambda v_{1i})
  = \lambda (a_{1i} - \underline{f}_1(x) + \langle v_{1i}, x \rangle, v_{1i})
\end{equation}
one gets that $\underline{d} f = \lambda \underline{d} f_1$ and $\overline{d} f = \lambda \overline{d} f_1$, i.e.
$D f = [ \lambda \underline{d} f_1, \lambda \overline{d} f_1 ]$ is a global codifferential of $f$ associated with the DC
decomposition $f = \underline{f} - (- \overline{f})$ defined above.

Let now $\lambda < 0$. Define $\underline{f}(x) = \lambda \overline{f}_1(x)$ and 
$\overline{f}(x) = \lambda \underline{f}_1(x)$. Then taking into account the fact the negative of a convex function is
a concave function and vice versa one obtains that $f = \lambda f_1 = \underline{f} - (- \overline{f})$ is a DC
decomposition of the function $f$. By definition one has
$$
  \underline{f}(x) = \max_{j \in J_1} (\lambda b_{1j} + \langle \lambda w_{1j}, x \rangle), \quad
  \overline{f}(x) = \min_{i \in I_1} (\lambda a_{1i} + \langle \lambda v_{1i}, x \rangle)
$$
(recall that $\lambda < 0$). Hence applying \eqref{DefOfGlobalCodiff_PiecewiseAff} and the fact that
$$
  (\lambda b_{1j} - \underline{f}(x) + \langle \lambda w_{1j}, x \rangle, \lambda w_{1j})
  = \lambda (b_{1j} - \overline{f}_1(x) + \langle w_{1j}, x \rangle, w_{1j})
$$
one obtains that $\underline{d} f = \lambda \overline{d} f_1$ and $\overline{d} f = \lambda \underline{d} f_1$,
i.e. $D f = [ \lambda \overline{d} f_1, \lambda \underline{d} f_1 ]$ is a global codifferential of $f$ associated with
the DC decomposition $f = \underline{f} - (- \overline{f})$ defined above.

\noindent\ref{Calc_Sum}. Define
\begin{equation} \label{DCdecomp_Sum}
  \underline{f}(x) = \sum_{m = 1}^p \underline{f}_m(x), \quad
  \overline{f}(x) = \sum_{m = 1}^p \overline{f}_m(x).
\end{equation}
Then $f = \underline{f} - (- \overline{f})$ is a DC decomposition of the function $f$ due to the fact that the sum of
convex/concave functions is a convex/concave function. Note that
$$
  \underline{f}(x) = \sum_{m = 1}^p \max_{i \in I_m} (a_{mi} + \langle v_{mi}, x \rangle )
  = \max_{(i_1,\ldots, i_p) \in I_1 \times \ldots \times I_p}
  \Big( \sum_{m = 1}^p a_{m i_m} + \Big\langle \sum_{m = 1}^p v_{m i_m}, x \Big\rangle \Big),
$$
and a similar equality holds true for $\overline{f}(x)$. Hence with the use of
\eqref{DefOfGlobalCodiff_PiecewiseAff}, and the fact that
\begin{multline*}
  \Big( \sum_{m = 1}^p a_{m i_m} - \underline{f}(x) + \Big\langle \sum_{m = 1}^p v_{m i_m}, x \Big\rangle, 
  \sum_{m = 1}^p v_{m i_m} \Big) \\
  = \sum_{m = 1}^p \big( a_{m i_m} - \underline{f}_m(x) + \langle v_{m i_m}, x \rangle, v_{m i_m} \big)
\end{multline*}
one obtains that $\underline{d} f(x) = \sum_{m = 1}^p \underline{d} f_m(x)$ and
$\overline{d} f(x) = \sum_{m = 1}^p \overline{d} f_m(x)$, i.e. 
$D f = [ \sum_{m = 1}^p \underline{d} f_m, \sum_{m = 1}^p \overline{d} f_m ]$ is a global codifferential of $f$
associated with the DC decomposition $f = \underline{f} - (- \overline{f})$ defined by \eqref{DCdecomp_Sum}.

\noindent\ref{Calc_Max}. Define
\begin{equation} \label{DCdecomp_Max}
  \underline{f}(x) = \max_{m \in M} \Big( \underline{f}_m(x) - \sum_{k \ne m} \overline{f}_m(x) \Big),	\quad
  \overline{f}(x) = \sum_{m = 1}^p \overline{f}_m(x).
\end{equation}
Note that the function $\underline{f}$ is convex, since the maximum and the sum of convex functions is convex, while the
function $\overline{f}$ is concave as the sum of concave functions. By definition 
$$
  f(x) = \max_{m \in M} f_m(x) = \max_{m \in M} \Big( \underline{f}_m(x) + \overline{f}_m(x) \Big).
$$
Adding and subtracting $\sum_{m = 1}^p \overline{f}_m$ one obtains that
$$
  f(x) = \max_{m \in M} \Big( \underline{f}_m(x) - \sum_{k \ne m} \overline{f}_m(x) \Big) 
  + \sum_{m = 1}^p \overline{f}_m(x) = \underline{f}(x) + \overline{f}(x),
$$
i.e. $f = \underline{f} - (- \overline{f})$ is a DC decomposition of the function $f$. Let us compute the global
codifferential of $f$ associated with this DC decomposition. 

From the proof of part \eqref{Calc_Sum} it follows that $\overline{d} f(x) = \sum_{m = 1}^p \overline{d} f_m(x)$ 
(cf. \eqref{DCdecomp_Sum} and \eqref{DCdecomp_Max}). Let us compute the global hypodifferential. By definition one has
\begin{align*}
  \underline{f}(x) &= \max_{m \in M} \Big( \underline{f}_m(x) - \sum_{k \ne m} \overline{f}_m(x) \Big) \\
  &= \max_{m \in M} \Big( \max_{i_m \in I_m} (a_{m i_m} + \langle v_{m i_m}, x \rangle)
  - \sum_{k \ne m} \min_{j_k^m \in J_k} (b_{k j_k^m} + \langle w_{k j_k^m}, x \rangle \Big) \\
  &= \max \Big( a_{m i_m} - \sum_{k \ne m} b_{k j_k^m} + 
  \Big\langle v_{m i_m} - \sum_{k \ne m} w_{k j_k^m}, x \Big\rangle \Big),
\end{align*}
where the last maximum is taken over all $i_m \in I_m$, $j_k^m \in J_k$, and $k, m \in M$. Hence and from
\eqref{DefOfGlobalCodiff_PiecewiseAff} it follows that the first coordinate of a vector from $\underline{d} f(x)$ has
the form
$$
  a_{m i_m} - \sum_{k \ne m} b_{k j_k^m} - \underline{f}(x)
  + \Big\langle v_{m i_m} - \sum_{k \ne m} w_{k j_k^m}, x \Big\rangle.
$$
Adding and subtracting $\underline{f}_m(x) - \sum_{k \ne m} \overline{f}_k(x)$, and taking into account the fact that
\begin{align*}
  \underline{f}_m(x) - \sum_{k \ne m} \overline{f}_k(x) - \underline{f}(x) 
  &= \underline{f}_m(x) + \overline{f}_m(x) - \sum_{k = 1}^p \overline{f}_k(x) - \underline{f}(x) \\
  &= f_m(x) - \overline{f}(x) - \underline{f}(x) = f_m(x) - f(x)
\end{align*}
one obtains that
\begin{multline*}
  a_{m i_m} - \sum_{k \ne m} b_{k j_k^m} - \underline{f}(x)
  + \Big\langle v_{m i_m} - \sum_{k \ne m} w_{k j_k^m}, x \Big\rangle \\
  = \big( f_m(x) - f(x) \big) + \Big( a_{m i_m} - \underline{f}_m(x) + \langle v_{m i_m}, x \rangle \Big)
  - \sum_{k \ne m} \Big( b_{k j_k^m} - \overline{f}_k(x) + \langle w_{k j_k^m}, x \rangle \Big).
\end{multline*}
Hence with the use of \eqref{DefOfGlobalCodiff_PiecewiseAff} one gets that
\begin{multline} 
  \underline{d} f(x) = \co\Big\{ (f_m(x) - f(x), 0) 
  + \big( a_{m i_m} - \underline{f}_m(x) + \langle v_{m i_m}, x \rangle, v_{m i_m} \big) \\
  - \sum_{k \ne m} \big( b_{k j_k^m} - \overline{f}_k(x) + \langle w_{k j_k^m}, x \rangle, w_{k j_k^m} \big)
  \Bigm| i_m \in I_m, \: j_k^m \in J_k, \: k \in M, \: m \in M \Big\}. \label{HypodiffMax_Direct}
\end{multline}
From the fact that by definition
\begin{align*}
  \big( a_{m i_m} - \underline{f}_m(x) + \langle v_{m i_m}, x \rangle, v_{m i_m} \big) &\in \underline{d} f_m(x), \\
  \big( b_{k j_k^m} - \overline{f}_k(x) + \langle w_{k j_k^m}, x \rangle, w_{k j_k^m} \big) 
  &\in \overline{d} f_k(x)
\end{align*}
it follows that
\begin{equation} \label{HypodiffMax_Direct_Incl}
  \underline{d} f(x) \subseteq
  \co\Big\{ (f_m(x) - f(x), 0) + \underline{d} f_m(x) - \sum_{k \ne m} \overline{d} f_k(x) \Bigm| m \in M \Big\}.
\end{equation}
To prove the converse inclusion fix $m \in M$, and note that taking the convex hull in \eqref{HypodiffMax_Direct} only
over $i_m \in I_m$ one obtains that
$$
  (f_m(x) - f(x), 0) + \underline{d} f_m(x)
  - \sum_{k \ne m} \big( b_{k j_k^m} - \overline{f}_k(x) + \langle w_{k j_k^m}, x \rangle, w_{k j_k^m} \big)
  \subseteq \underline{d} f(x).
$$
Now, taking consecutively the convex hull over all $j_k^m \in J_k$ for each $k \ne m$ one gets that
$$
  (f_m(x) - f(x), 0) + \underline{d} f_m(x) - \sum_{k \ne m} \overline{d} f_k(x) \subseteq \underline{d} f(x).
$$
Finally, taking the convex hull over all $m \in M$ one obtains that the inclusion opposite to
\eqref{HypodiffMax_Direct_Incl} is valid, which implies the desired result.

\noindent\ref{Calc_Min}. Define
$$
  \underline{f}(x) = \sum_{m = 1}^p \underline{f}_m(x),	\quad
  \overline{f}(x) = \min_{m \in M} \Big( \overline{f}_m(x) - \sum_{k \ne m} \underline{f}_m(x) \Big).
$$
Clearly, the function $\underline{f}$ is convex, while the function $\overline{f}$ is concave. By definition 
$$
  f(x) = \min_{m \in M} f_m(x) = \min_{m \in M} \Big( \underline{f}_m(x) + \overline{f}_m(x) \Big).
$$
Adding and subtracting $\sum_{m = 1}^p \underline{f}_m$ one obtains that
$$
  f(x) = \sum_{m = 1}^p \overline{f}_m 
  + \min_{m \in M} \Big( \overline{f}_m(x) - \sum_{k \ne m} \underline{f}_m(x) \Big) 
  = \underline{f}(x) + \overline{f}(x),
$$
i.e. $f = \underline{f} - (- \overline{f})$ is a DC decomposition of the function $f$. Computing the global
codifferential of the function $f$ associated with this DC decomposition in the same way as in part (\ref{Calc_Max}) one
obtains the required result (alternatively, one can rewrite $f = - \max_{m \in M} (- f_m)$, and consecutively
apply part (\ref{Calc_ScalarMultiplication}) with $\lambda = -1$, part (\ref{Calc_Max}), and part
(\ref{Calc_ScalarMultiplication}) with $\lambda = -1$ again to obtain exactly the same result).
\end{proof}

\begin{remark} \label{Remark_DCdecomp_Equiv_Codiff}
{(i)~Note that with the use of the proposition above one can compute DC decomposition \eqref{DCdecomp_PiecewiseAffFunc}
of a piecewise affine function (see~\cite{Angelov} for more details). Namely, suppose that a global
codifferential $D f(0)$ of $f$ at zero is known, $\underline{d} f(0) = \co\{ (a_i, v_i) \mid 1 \le i \le l \}$, and 
$\overline{d} f(0) = \co\{ (b_j, w_j) \mid 1 \le j \le s \}$. Applying \eqref{PiecewiseAffFunc_CodiffDecomp} with 
$x = 0$ one obtains that
\begin{align*}
  f(\Delta x) - f(0) &= \max_{(a, v) \in \underline{d} f(0)} (a + \langle v, \Delta x \rangle)
  + \min_{(b, w) \in \overline{d} f(0)} (b + \langle w, \Delta x \rangle) \\
  &= \max_{1 \le i \le l} (a_i + \langle v_i, \Delta x \rangle)
  + \min_{1 \le j \le s} (b_j + \langle w_j, \Delta x \rangle).
\end{align*}
Define
\begin{equation} \label{DC_decomp_Via_GlobalCodiff}
  \underline{f}(x) = \max_{1 \le i \le l} (a_i + f(0) + \langle v_i, x \rangle), \quad
  \overline{f}(x) = \min_{1 \le j \le s} (b_j + \langle w_j, x \rangle).
\end{equation}
Then $f = \underline{f} - (- \overline{f})$ is a DC decomposition of the function $f$, i.e. there is a
one-to-one correspondence between DC decompositions and global codifferentials of piecewise affine functions. Let us
also note that from the definition of global codifferential \eqref{DefOfGlobalCodiff_PiecewiseAff}, and the equality
$$
  (a_i - \underline{f}(y) + \langle v_i, y \rangle, v_i) = 
  (a_i - \underline{f}(x) + \langle v_i, x \rangle, v_i) 
  + ( \underline{f}(x) - \underline{f}(y) + \langle v_i, y - x \rangle, 0 )
$$
it follows that
\begin{equation} \label{RecomputingGlobalCodiff}
  \underline{d} f(y) 
  = \big\{ (a + \underline{f}(x) - \underline{f}(y) + \langle v, y - x \rangle, v) \in \mathbb{R}^{d + 1}
  \bigm| (a, v) \in \underline{d} f(x) \big\}
\end{equation}
for all $x, y \in \mathbb{R}^d$, and a similar equality holds true for $\overline{d} f(\cdot)$. Thus, one can easily
compute $D f(y)$ for any $y$, if $D f(x)$ for some $x$ is known.
}

\noindent{(ii)~It should be mentioned that the proposition above allows one to compute a global codifferential of a
piecewise affine function $f$ without computing its DC decomposition. Nevertheless, in order to avoid rather lengthy
computations at every point $x$ it seems most reasonable to compute $D f(0)$ first, then compute a DC decomposition of
$f$ with the use of \eqref{DC_decomp_Via_GlobalCodiff}, and, finally, utilise \eqref{RecomputingGlobalCodiff} to
compute $D f(x)$ at any point $x \in \mathbb{R}^d$.
}

\noindent{(iii)~Let us note that the proper choice of a global codifferential of the affine function 
$a + \langle v, x \rangle$ allows one to simplify the computation of a global codifferential of a piecewise affine
function. The first global codifferential from part (\ref{Calc_AffineFunc}) of Prop.~\ref{Prp_GlobalCodiffCalc} is
more suitable for the computation of a global codifferential of the maximum of affine functions, while the second one is
more suitable in the case of the minimum. Indeed, if $f = \max_{1 \le i \le l} f_i$, where
$f_i(x) = a_i + \langle v_i, x \rangle$, then applying part \eqref{Calc_Max} of Prop.~\ref{Prp_GlobalCodiffCalc} one
obtains
$$
  \underline{d} f(x) = \co\big\{ (a_i + \langle v_i, x \rangle - f(x), v_i) \bigm| 1 \le i \le l \big\}, \quad
  \overline{d} f(x) = \{ (0, 0) \}
$$
for $D f_i(\cdot) = [ \{ (0, v_i) \}, \{ (0, 0) \} ]$, while
$$
  \underline{d} f(x) = \co\Big\{ \Big(a_i + \langle v_i, x \rangle - f(x), - \sum_{k \ne i} v_k  \Big) 
  \Bigm| 1 \le i \le l\Big\},  \quad \overline{d} f(x) = \Big\{ \Big(0, \sum_{i = 1}^l v_i \Big) \Big\}
$$
for $D f_i(\cdot) = [ \{ (0, 0) \}, \{ (0, v_i) \} ]$.
}
\end{remark}

Let us derive new global optimality conditions for a piecewise affine function in terms of its global codifferential.

\begin{theorem} \label{Thrm_GlobOptCond}
Let $f \colon \mathbb{R}^d \to \mathbb{R}$ be a piecewise affine function of the form \eqref{DCdecomp_PiecewiseAffFunc},
$D f$ be its global codifferential mapping, and $x^* \in \mathbb{R}^d$ be a given point. Suppose also that $f$ is
bounded below, and for any $j \in J$ define 
$z_j = (b_j - \overline{f}(x^*) + \langle w_j, x^* \rangle, w_j) \in \overline{d} f(x^*)$, and 
\begin{equation} \label{GlobOptCond_MinCodiffGrad}
  \big\{ (a_j, v_j) \big\} = 
  \argmin\big\{ \| (a, v) \|^2 \bigm| (a, v) \in \underline{d} f(x^*) + z_j \big\}.
\end{equation}
Then $x^*$ is a point of global minimum of the function $f$ if and only if for any $j \in J$, one has 
$a_j \ge 0$ or, equivalently, for any $j \in J$ either $0 \in \underline{d} f(x^*) + z_j$ or $a_j > 0$.
\end{theorem}

\begin{proof}
Applying equality \eqref{PiecewiseAffFunc_CodiffDecomp} with $x = x^*$ and $\Delta x = x - x^*$, and the definition of
global codifferential \eqref{DefOfGlobalCodiff_PiecewiseAff} one obtains that
\begin{multline*}
  f(x) - f(x^*) = \max_{(a, v) \in \underline{d} f(x^*)} (a + \langle v, x - x^* \rangle) \\
  + \min_{j \in J} (b_j - \overline{f}(x^*) + \langle w_j, x^* \rangle + \langle w_j, x - x^* \rangle) \\
  = \min_{j \in J} \max_{(a, v) \in \underline{d} f(x^*)} 
  \Big( a + b_j - \overline{f}(x^*) + \langle w_j, x^* \rangle + \langle v + w_j, x - x^* \rangle \Big)
\end{multline*}
for all $x \in \mathbb{R}^d$. Hence for any $x \in \mathbb{R}^d$ one has
\begin{equation} \label{CoexhausterRepresentation}
  f(x) - f(x^*) = \min_{j \in J} \max_{(a, v) \in \underline{d} f(x^*) + z_j} (a + \langle v, x - x^* \rangle)
  = \min_{j \in J} g_j(x),
\end{equation}
where $g_j(x) = \max_{(a, v) \in \underline{d} f(x^*) + z_j} (a + \langle v, x \rangle)$. Therefore, $x^*$ is point of
global minimum of $f$ if and only if for any $j \in J$ the function $g_j$ is nonnegative. Note that each function $g_j$
is bounded below due to the facts that $g_j(x) \ge f(x) - f(x^*)$ for all $x \in \mathbb{R}^d$ (see
\eqref{CoexhausterRepresentation}), and $f$ is bounded below. Consequently, applying
Lemma~\ref{Lemma_PolyhedralFunc_Nonnegative} and Corollary~\ref{Corollary_PolyhedralFunc_Nonnegative} to the functions
$g_j$ one obtains the desired result.
\end{proof}

The necessary and sufficient conditions for global optimality in terms of global codifferential from
Theorem~\ref{Thrm_GlobOptCond} along with the proof of Lemma~\ref{Lemma_PolyhedralFunc_Nonnegative} allow one to get a
new perspective on the method of codifferential descent. As it was noted above, a function is codifferentiable if and
only if its increment can be locally approximated by a DC function. In most applications a codifferential of a nonsmooth
function is a pair of convex polytopes, i.e. the increment of this function can be locally approximated by a piecewise
affine function. In a sense, in each iteration of the method of codifferential descent one verifies whether the global
optimality conditions from Theorem~\ref{Thrm_GlobOptCond} are satisfied for a local piecewise affine approximation of
the objective function, and then utilises the ``global descent'' directions $- v_j$ of the approximation (see
\eqref{GlobOptCond_MinCodiffGrad} and the proof of the first part of Lemma~\ref{Lemma_PolyhedralFunc_Nonnegative}) as
search directions for the objective function. In the case when the objective function itself is piecewise affine, and
its global codifferential mapping is known, one can propose a natural modification of the MCD in which instead of
performing the line search one utilises the first component of the vector $(a_j, v_j)$ in order to define the step size.

\subsection{The method of global codifferential descent}

Let $f \colon \mathbb{R}^d \to \mathbb{R}$ be a piecewise affine function of the form \eqref{DCdecomp_PiecewiseAffFunc},
and $D f$ be its global codifferential mapping (see \eqref{DefOfGlobalCodiff_PiecewiseAff}). For any 
$x \in \mathbb{R}^d$ and $j \in J$ denote
\begin{gather} \label{HyperGrad}
  z_j(x) = (b_j - \overline{f}(x) + \langle w_j, x \rangle, w_j) \in \overline{d} f(x), \\
  \big\{ (a_j(x), v_j(x)) \big\} = 
  \argmin\big\{ \| (a, v) \|^2 \bigm| (a, v) \in \underline{d} f(x) + z_j(x) \big\}. \label{ClosestPoint_MCD}
\end{gather}
Suppose that $x$ is not a point of global minimum of the function $f$, and choose an arbitrary 
$j \in J$. Applying the necessary and sufficient condition for a minimum of a convex function on a convex set
\cite[Proposition~II.2.1]{EkelandTemam} one obtains that
$$
  a_j(x) (a - a_j(x)) + \langle v_j(x), v - v_j(x) \rangle \ge 0 
  \quad \forall (a, v) \in \underline{d} f(x) + z_j(x).
$$
If $a_j(x) < 0$, then dividing this inequality by $a_j(x)$, taking the maximum over all 
$(a, v) \in \underline{d} f(x) + z_j(x)$, and applying \eqref{CoexhausterRepresentation} one obtains
\begin{equation} \label{MGCD_GlobalDescentDirections}
  f\left( x + \frac{1}{a_j(x)} v_j(x) \right) - f(x) 
  \le - \frac{1}{|a_j(x)|} \| (a_j(x), v_j(x)) \|^2 < 0.
\end{equation}
If $a_j(x) = 0$, but $v_j(x) \ne 0$, then $\langle v, - v_j(x) \rangle \le - \| v_j(x) \|^2$ for any 
$(a, v) \in \underline{d} f(x) + z_j(x)$, which with the use of \eqref{CoexhausterRepresentation} implies that
\begin{equation} \label{MGCD_UnboundedBelow}
  f(x - \alpha v_j(x)) - f(x) \le \max_{(a, v) \in \underline{d} f(x) + z_j(x)} a - \alpha \| v_j(x) \|^2,
\end{equation}
and the function $f$ is unbounded below. Thus, if $a_j(x) = 0$, and $f$ is bounded below, then $v_j(x) = 0$.

Finally, if $a_{j}(x) > 0$, then the set $\underline{d} f(x) + z_j(x)$ is of no use to the optimization process.
Indeed, by Lemma~\ref{Lemma_PolyhedralFunc_Nonnegative} one has
\begin{equation} \label{NonnegativeOfConvexPiece}
  \max_{(a, v) \in \underline{d} f(x) + z_j(x)} (a + \langle v, y \rangle) \ge 0 
  \quad \forall y \in \mathbb{R}^d.
\end{equation}
Applying \eqref{CoexhausterRepresentation} one gets that
$$
  f(y) - f(x) = \min_{k \in J} \max_{(a, v) \in \underline{d} f(x) + z_k(x)} (a + \langle v, y - x \rangle)
  \quad \forall y \in \mathbb{R}^d.
$$ 
From \eqref{NonnegativeOfConvexPiece} it follows that for any $y$ such that $f(y) < f(x)$ the minimum in this
equality cannot be achieved for $k = j$. Therefore
$$
  f(y) - f(x) = \min_{k \in J \setminus \{ j \}} 
  \max_{(a, v) \in \underline{d} f(x) + z_k(x)} (a + \langle v, y - x \rangle)
$$
for any $y \in \mathbb{R}$ such that $f(y) < f(x)$.
In other words, the index $j$ and the corresponding vector $(b_j, w_j)$ are not needed to compute $f(y)$ for any 
$y \in \mathbb{R}^d$ satisfying the inequality $f(y) < f(x)$.

Let us prove an even stronger statement. Namely, let us show that if $a_j(x) \ge 0$ for some $x \in \mathbb{R}^d$, then 
the index $j$ can be discarded from consideration.

\begin{lemma} \label{Lemma_MGCD_UselessVertices}
Let $f \colon \mathbb{R}^d \to \mathbb{R}$ be a piecewise affine function of the form \eqref{DCdecomp_PiecewiseAffFunc},
and $D f$ be its global codifferential mapping. Suppose that $f$ is bounded below, and for some $j \in J$ and 
$x \in \mathbb{R}^d$ one has $a_j(x) \ge 0$. Then $a_j(y) \ge 0$ for any $y \in \mathbb{R}^d$ such that $f(y) \le f(x)$.
\end{lemma}

\begin{proof}
For any $\Delta y, y \in \mathbb{R}^d$ denote 
$$
  g_j(\Delta y, y) = \max_{(a, v) \in \underline{d} f(y) + z_j(y)} (a + \langle v, \Delta y \rangle).
$$
Applying \eqref{PiecewiseAffFunc_CodiffDecomp} and \eqref{DefOfGlobalCodiff_PiecewiseAff} one gets that
\begin{multline*}
  f(y + \Delta y) - f(y) = \max_{(a, v) \in \underline{d} f(y)} (a + \langle v, \Delta y \rangle)
  + \min_{(b, w) \in \overline{d} f(y)} (b + \langle w, \Delta y \rangle) \\
  \le \max_{(a, v) \in \underline{d} f(y)} (a + \langle v, \Delta y \rangle)
  + b_j - \overline{f}(y) + \langle w_j, y \rangle + \langle w_j, \Delta y \rangle
  = g_j(\Delta y, y).
\end{multline*}
for any $\Delta y, y \in \mathbb{R}^d$. Hence taking into account the fact that $f$ is bounded below one obtains that
the function $g(\cdot, y)$ is bounded below for any $y \in \mathbb{R}^d$. Furthermore, note that by the definition of
$\underline{d} f(\cdot)$ and $z_j(\cdot)$ (see~\eqref{DefOfGlobalCodiff_PiecewiseAff} and \eqref{HyperGrad}) one has
\begin{equation} \label{PiecewiseAffUCA}
  g_j(\Delta y, y) = \max_{i \in I} \big(a_i + \langle v_i, y \rangle - \underline{f}(y) + 
  \langle v_i, \Delta y \rangle
  + b_j + \langle w_j, y \rangle - \overline{f}(y) + \langle w_j, \Delta y \rangle \big)
\end{equation}
for all $\Delta y, y \in \mathbb{R}^d$.

From Corollary~\ref{Corollary_PolyhedralFunc_Nonnegative} and the fact that $a_j(x) \ge 0$ it follows that the function
$g(\cdot, x)$ is nonnegative. Hence with the use of \eqref{PiecewiseAffUCA} one obtains that for any 
$\Delta x \in \mathbb{R}^d$ there exists $i \in I$ such that
$$
  a_i + \langle v_i, x \rangle - \underline{f}(x) + \langle v_i, \Delta x \rangle
  + b_j + \langle w_j, x \rangle - \overline{f}(x) + \langle w_j, \Delta x \rangle \ge 0.
$$
Setting $\Delta x = y - x + \Delta y$, and taking into account the fact that 
$f(x) = \underline{f}(x) + \overline{f}(x)$ by definition (see \eqref{DCdecomp_PiecewiseAffFunc} and
\eqref{ConvexConcaveParts_PiecewiseAffine}) one gets that for any $\Delta y, y \in \mathbb{R}^d$ there exists $i \in I$
such that
$$
  a_i + \langle v_i, y \rangle + \langle v_i, \Delta y \rangle
  + b_j + \langle w_j, y \rangle + \langle w_j, \Delta y \rangle \ge f(x).
$$
Subtracting $f(y) = \underline{f}(y) + \overline{f}(y)$ from both sides of this inequality one obtains that 
for any $y, \Delta y \in \mathbb{R}^d$ there exists $i \in I$ such that
$$
  a_i + \langle v_i, y \rangle - \underline{f}(y) + \langle v_i, \Delta y \rangle
  + b_j + \langle w_j, y \rangle - \overline{f}(y) + \langle w_j, \Delta y \rangle \ge f(x) - f(y).
$$
Taking the maximum over all $i \in I$, and applying \eqref{PiecewiseAffUCA} one gets that 
$g_j(\Delta y, y) \ge f(x) - f(y)$ for all $\Delta y, y \in \mathbb{R}^d$. Hence the function $g_j(\cdot, y)$ is
nonnegative for any $y$ such that $f(y) \le f(x)$, which by Corollary~\ref{Corollary_PolyhedralFunc_Nonnegative} and 
the definition of $a_j(y)$ (see~\eqref{ClosestPoint_MCD}) implies that $a_j(y) \ge 0$ for all such $y$.
\end{proof}

Now, we can introduce a modification of the method of codifferential descent for minimizing piecewise affine functions
of the form \eqref{DCdecomp_PiecewiseAffFunc}, which we call the method of global codifferential descent (MGCD). The
scheme of this method is given in Algorithm~\ref{algorithm_MGCD}. 

\begin{algorithm}[t] \label{algorithm_MGCD}
\caption{The method of global codifferential descent (MGCD).}

\noindent\textit{Step}~1. {Choose a starting point $x_0 \in \mathbb{R}^d$, and set $M = J = \{ 1, \ldots, s \}$ and 
$n := 0$.
}

\noindent\textit{Step}~2. {Compute $\underline{d} f(x_n)$ and $z_j(x_n)$ for all $j \in M$.}

\noindent\textit{Step}~3. {For any $j \in M$ compute $(a_j(x_n), v_j(x_n)) \in \mathbb{R}^{d + 1}$ by
solving
$$
  \min \| (a, v) \|^2	\quad \text{s.t. } (a, v) \in \underline{d} f(x_n) + z_j(x_n).
$$
If $a_j(x_n) \ge 0$, then $M := M \setminus \{ j \}$.
}

\noindent\textit{Step}~4. {If $M = \emptyset$, then \textit{stop}. Otherwise, compute $j(n) \in M$ by solving
$$
  \min_{j \in M} f\left( x_n + \frac{1}{a_j(x_n)} v_j(x_n) \right).
$$
Set $x_{n + 1} = x_n + [a_{j(n)}(x_n)]^{-1} v_{j(n)}(x_n)$, and $n := n + 1$, and go to Step~2.
}
\end{algorithm}

Let a sequence $\{ x_n \}$ be generated by the MGCD. Observe that from \eqref{MGCD_GlobalDescentDirections} it follows
that for any $n \in \mathbb{N}$ either $f(x_{n + 1}) < f(x_n)$ or $M = \emptyset$. Hence, in particular, if 
$a_j(x_n) \ge 0$ for some $j \in J$ and $n \in \mathbb{N}$, then 
\begin{equation} \label{MGCD_FiniteTermin_Discuss}
  a_j(x_k) \ge 0  \quad \forall k \ge n
\end{equation}
by Lemma~\ref{Lemma_MGCD_UselessVertices}. Therefore, if the MGCD terminates in an iteration $n$ (i.e. $M = \emptyset$
for some $n \in \mathbb{N}$), then $a_j(x_n) \ge 0$ for all $j \in J$, which by Theorem~\ref{Thrm_GlobOptCond} implies
that $x_n$ is a point of global minimum of the function $f$. Below, we prove that the MGCD always terminates in a finite
number of steps, i.e. it finds a global minimizer of a nonconvex piecewise affine function in a finite number of steps.

At first, let us explain the idea behind the proof of this result, which also illuminates the way each step of the MGCD
is performed. Suppose for the sake of simplicity that the function $f$ is convex 
(i.e. $\overline{f}(x) \equiv \{ 0 \}$, see \eqref{ConvexConcaveParts_PiecewiseAffine}). The hypodifferential
$\underline{d} f(x_n)$ is a convex polytope in $\mathbb{R}^{d + 1}$. By \eqref{HypodiffNonnegative} one has $a \le 0$
for any $(a, v) \in \underline{d} f(x_n)$, and $\max_{(a, v) \in \underline{d} f(x_n)} a = 0$. Thus, the set 
$\{ (a, v) \in \mathbb{R}^{d + 1} \mid a = 0 \} \cap \underline{d} f(x_n)$ is a nonempty face of $\underline{d} f(x_n)$
(by \eqref{MGCD_UnboundedBelow} this face is proper, i.e. it does not coincide with $\underline{d} f(x_n)$, since
otherwise $f$ is unbounded below). We call it \textit{the active face} of the polytope $\underline{d} f(x_n)$. It is
easy to see that the subdifferential $\partial f(x_n)$ is exactly the set of those $v$ for which $(0, v)$ belongs to the
active face of $\underline{d} f(x_n)$.

\begin{figure}[t]
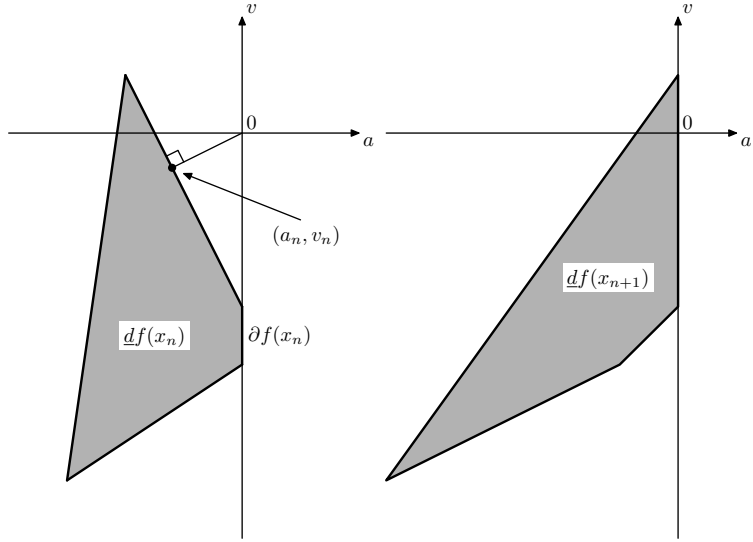

\centering
\includegraphics[width=0.4\linewidth]{Example1}
\includegraphics[width=0.4\linewidth]{Example2}
\caption{The transformation of the global codifferential over one step of the MGCD: $\underline{d} f(x_n)$ (left figure)
and $\underline{d} f(x_{n + 1})$ (right figure). Note that all points shift only horizontally, i.e. along the
$a$-axis (see~\eqref{DefOfGlobalCodiff_PiecewiseAff}).}
\label{Fig_HypodiffTransformation}
\end{figure}

The point
\begin{equation} \label{ClosestPointHypodiff_ConvexCase}
  \{ (a_n, v_n) \} = \argmin\Big\{ \| (a, v) \| \Bigm| (a, v) \in \underline{d} f(x_n) \Big\}
\end{equation}
lies on a face $F$ of $\underline{d} f(x_n)$, which is not active, since otherwise, $f$ is unbounded below by
\eqref{MGCD_UnboundedBelow}. When one performs one iteration of the MGCD, the polytope $\underline{d} f(x_n)$
transforms, and, as we will show in the proof below, the face $F$ becomes the active face of the polytope 
$\underline{d} f(x_{n + 1})$. Thus, the projection $(a_n, v_n)$ belongs to a face of the hypodifferential, which
becomes active on the next iteration (see~Fig.~\ref{Fig_HypodiffTransformation}). 

Bearing these observations in mind one can prove the finite convergence of the MGCD by showing that in a finite
number of iterations the projection $(a_n, v_n)$ belongs to a face of $\underline{d} f(x_n)$ that intersects the axis 
$\{ (a, 0) \in \mathbb{R}^{d + 1} \mid a \in \mathbb{R} \}$. Then $0 \in \partial f(x_{n + 1})$, and the proof is
complete. In the case, when the function $f$ is not convex, a similar argument allows one to prove that in a finite
number of iterations an index $j(n)$ is discarded. Repeating the same argument $s$ times one can verify that in a finite
number of iterations all indices are discarded, and the MGCD terminates.

\begin{theorem} \label{Thrm_MGCD_PiecewiseAffine}
Let $f \colon \mathbb{R}^d \to \mathbb{R}$ be a bounded below piecewise affine function. Then $f$ attains a global
minimum, and the MGCD finds a point of global minimum of this function in a finite number of steps.
\end{theorem}

\begin{proof}
Let $\{ x_n \}$ be a possibly infinite sequence generated by the MGCD for the function $f$. Denote
$a_n = a_{j(n)}(x_n)$ and $v_n = v_{j(n)}(x_n)$, where the index $j(n)$ is computed on Step~4 of the MGCD. Note that
this definition of $(a_n, v_n)$ coincides with \eqref{ClosestPointHypodiff_ConvexCase}, if $\overline{f}(x) \equiv 0$,
since in this case $z_j(x) \equiv 0$ for all $j$ (see Steps~3 and 4 of the MGCD,
\eqref{ConvexConcaveParts_PiecewiseAffine} and \eqref{HyperGrad}). 

From Theorem~\ref{Thrm_GlobOptCond} it follows that if $x_n$ is not a global minimizer of $f$, then there exists 
$j \in J$ such that $a_j(x_n) < 0$, and
\begin{align}
  f(x_{n + 1}) \le f\left( x_n + \frac{1}{a_j(x_n)} v_j(x_n) \right) 
  &\le f(x_n) - \frac{1}{|a_j(x_n)|} \big\| (a_j(x_n), v_j(x_n)) \big\|^2  \notag\\
  &= f(x_n) - |a_j(x_n)| - \frac{1}{|a_j(x_n)|} \| v_j(x_n) \|^2 \label{MGCD_DecayOverOneIteration}
\end{align}
(see \eqref{MGCD_GlobalDescentDirections} and Step~4 of the MGCD). Note that
$$
  - |a_j(x_n)| - \frac{1}{|a_j(x_n)|} \| v_j(x_n) \|^2 
  \le \begin{cases}
    - 1, & \text{if } |a_j(x_n)| \ge 1, \\
    - \| v_j(x_n) \|^2, & \text{otherwise.}
  \end{cases}
$$
Hence, if $x_n$ is a not a point of global minimum of $f$, then 
\begin{equation} \label{MGCD_DecayEstimate}
  f(x_{n + 1}) - f(x_0) \le - \sum_{k = 0}^n 
  \left( |a_k| + \frac{1}{|a_k|} \big\| v_k \big\|^2 \right)
  \le - \sum_{k = 0}^n \min\big\{ 1, \| v_k \|^2 \big\}.
\end{equation}
Denote by $\mathcal{E}$ the family of all convex sets $C \subset \mathbb{R}^d$ such that $0 \notin C$, and
$$
  C = \co\{ v_{i_1}, \ldots, v_{i_k}  \} + w_j
$$
for some $i_1, \ldots, i_k \in I$, $1 \le k \le l$, and $j \in J$, where the vectors $v_i$ and $w_j$ are from the
DC decomposition of the function $f$ (see~\eqref{DCdecomp_PiecewiseAffFunc}). Clearly, $\mathcal{E}$ is a finite family
of compact convex sets, and $\theta = \min_{C \in \mathcal{E}} \min_{v \in C} \| v \|^2 > 0$. 

Denote $f^* = \inf_{x \in \mathbb{R}^d} f(x) > - \infty$, and 
$n^* = \lfloor (f(x_0) - f^*) / \min\{ \theta, 1 \} \rfloor + 1$ (here $\lfloor t \rfloor$ is the greatest integer less
than or equial to $t \in \mathbb{R}$). From \eqref{MGCD_DecayEstimate} it follows that there exists $n \le n^*$ such
that either the MGCD terminates at the step $n$ or $a_n < 0$ and $\| v_n \|^2 < \theta$. 

Suppose that $x_n$ is not a global minimizer of $f$. By definition $(a_n, v_n)$ belongs to the convex polytope
$\underline{d} f(x_n) + z_{j(n)}(x_n)$ (see Step~3 of the MGCD). Any convex polytope is equal to the disjoint union of
the relative interiors of its faces, i.e. the relative interiors of all faces of a convex polytope are pairwise
disjoint, and the polytope is equal to the union of these relative interiors (see~\cite{Ziegler}, p.~61). Therefore,
$(a_n, v_n)$ belongs to the relative interior $\relint F$ of a face $F$ of $\underline{d} f(x_n) + z_{j(n)}(x_n)$. 

With the use of the necessary and sufficient condition for a minimum of a convex function on a convex set
\cite[Proposition~II.2.1]{EkelandTemam} one obtains that
\begin{equation} \label{FiniteConv_SeparationThrm}
  a_n a + \left\langle v_n, v \right\rangle \ge \| (a_n, v_n) \|^2
  \quad \forall (a, v) \in \underline{d} f(x_n) + z_{j(n)}(x_n),
\end{equation}
and this inequality turns into an equality when $(a, v) = (a_n, v_n)$. By \cite[Prop.~2.3]{Ziegler} the face $F$ is
itself a polytope. Consequently, applying the characterization of relative interior points of a convex polytope
\cite[Lemma~2.9]{Ziegler} and the fact that $(a_n, v_n) \in \relint F$ one gets that 
\begin{equation} \label{EqualityOnFace}
  a_n a + \left\langle v_n, v \right\rangle = \| (a_n, v_n) \|^2 \quad \forall (a, v) \in F
\end{equation}
Note also that the face $F$ is a polytope whose vertices are vertices of $\underline{d} f(x_n) + z_{j(n)}(x_n)$ as well 
{\cite[Prop.~2.3]{Ziegler}}. Therefore 
$$
  F = \co\{ (a_{i_r} + \langle v_{i_r}, x_n \rangle - \underline{f}(x_n), v_{i_r}) \mid 1 \le r \le k \} 
  + z_{j(n)}(x_n)
$$
for some $i_1, \ldots, i_k \in I$ and $1 \le k \le l$ (see~\eqref{DefOfGlobalCodiff_PiecewiseAff}). From the
definition of $\theta$, and the facts that  $(a_n, v_n) \in F$ and $\| v_n \|^2 < \theta$ it follows that 
$F \cap (\mathbb{R} \times \{ 0 \}) \ne \emptyset$. 

Introduce the convex function 
\begin{equation} \label{AboutToBeDiscarded}
  g_n(x) = \max_{(a, v) \in \underline{d} f(x_{n + 1}) + z_{j(n)}(x_{n + 1})} (a + \langle v, x \rangle).
\end{equation}
Let us verify that $0 \in \partial g_n(0)$. Indeed, by the definition of $z_j(x)$ (see~\eqref{HyperGrad}) one has
\begin{align*}
  z_{j(n)}(x_{n + 1}) 
  &= \big( b_{j(n)} - \overline{f}(x_{n + 1}) + \langle w_{j(n)}, x_{n + 1} \rangle, w_{j(n)} \big) \\
  &= z_{j(n)}(x_n) 
  + \big( \overline{f}(x_n) - \overline{f}(x_{n + 1}) + \langle w_{j(n)}, x_{n + 1} - x_n \rangle, 0 \big).
\end{align*}
Similarly, by \eqref{RecomputingGlobalCodiff} one has
$$
  \underline{d} f(x_{n + 1}) 
  = \big\{ (a + \underline{f}(x_n) - \underline{f}(x_{n + 1}) + \langle v, x_{n + 1} - x_n \rangle, v) 
  \in \mathbb{R}^{d + 1} \bigm| (a, v) \in \underline{d} f(x_n) \big\}
$$
Consequently, applying the equality $f(x) = \underline{f}(x) + \overline{f}(x)$ (see
\eqref{DCdecomp_PiecewiseAffFunc} and \eqref{ConvexConcaveParts_PiecewiseAffine}) one obtains that
\begin{multline*}
  \underline{d} f(x_{n + 1}) + z_{j(n)}(x_{n + 1}) \\
  = \big\{ (a + \langle v, x_{n + 1} - x_n \rangle - f(x_{n + 1}) + f(x_n), v) \bigm| 
  (a, v) \in \underline{d} f(x_n) + z_{j(n)} (x_n) \big\}.
\end{multline*}
Therefore
$$
  g_n(0) = \max_{(a, v) \in \underline{d} f(x_n) + z_{j(n)} (x_n)} (a + \langle v, x_{n + 1} - x_n \rangle) 
  - f(x_{n + 1}) + f(x_n).
$$
Hence taking into account \eqref{FiniteConv_SeparationThrm}, \eqref{EqualityOnFace}, and the facts that 
$x_{n + 1} - x_n = a_n^{-1} v_n$ and $a_n < 0$ one gets that
\begin{equation} \label{NonnegativityOnNextStep}
  g_n(0) = \frac{1}{a_n} \| (a_n, v_n) \|^2 - f(x_{n + 1}) + f(x_n) \ge 0,
\end{equation}
where the last inequality follows from \eqref{MGCD_DecayOverOneIteration}. Furthermore, the maximum in the definition of
$g_n(0)$ is attained at the points $(a + \langle v, x_{n + 1} - x_n \rangle - f(x_{n + 1}) + f(x_n), v)$ with 
$(a, v) \in F$. Consequently, one has $\{ v \mid \exists (a, v) \in F \} \subseteq \partial g_n(0)$, which implies that
$0 \in \partial g_n(0)$ (since $F \cap (\mathbb{R} \times \{ 0 \}) \ne \emptyset$), i.e. $0$ is the point of global
minimum of the function $g_n(x)$. Hence and from \eqref{NonnegativityOnNextStep} it follows that the function $g_n$ is
nonnegative. Taking into account \eqref{AboutToBeDiscarded}, and applying
Corollary~\ref{Corollary_PolyhedralFunc_Nonnegative} one obtains that $a_{j(n)}(x_{n + 1}) \ge 0$ 
(see Step~3 of the MGCD). Therefore the index $j(n)$ is discarded by the MGCD, and by
Lemma~\ref{Lemma_MGCD_UselessVertices} one has $a_{j(n)}(x_k) \ge 0$ for all $k \ge n + 1$.

Thus, there exists $n_1 \le n^*$ such that the MGCD discards the index $j(n_1)$ in the $(n_1 + 1)$th iteration. Recall 
that $n^* = \lfloor (f(x_0) - f^*) / \min\{ \theta, 1 \} \rfloor + 1$. Taking into account \eqref{MGCD_DecayEstimate}
one obtains that there exists $n_2 \le n_1 + n^* \le 2 n^*$ such that either the MGCD terminates at the $n_2$th
iteration or $a_{n_2} < 0$, and $\| v_{n_2} \|^2 < \theta$. Arguing in the same way as above one can easily verify that
the MGCD discards the index $j(n_2)$ in the $(n_2 + 1)$th iteration, and $a_{j(n_2)}(x_k) \ge 0$ for all $k \ge n_2 +
1$. Repeating the same argument $s$ times one obtains that the MGCD discards all indices from the set $M$ in at most $s
n^*$ iterations, and, thus, terminates in a finite number of steps. Furthermore, if MGCD terminates at an $n$th
iteration, then, as it was pointed out above (see \eqref{MGCD_FiniteTermin_Discuss}), by
Lemma~\ref{Lemma_MGCD_UselessVertices} one has $a_j(x_n) \ge 0$ for all $j \in J$, which with the use of
Theorem~\ref{Thrm_GlobOptCond} implies that $x_n$ is a point of global minimum of the function $f$, and the proof is
complete.
\end{proof}

\begin{remark}
Note that the theorem above is valid for any method generating a sequence $\{ x_n \}$ such that 
for all $n \in \mathbb{N}$ one has $f(x_{n + 1}) \le f(x_n)$ and
\begin{equation} \label{Codiff_FiniteConvergCond}
  f(x_{n + 1}) \le f\left( x_n + \frac{1}{a_j(x_n)} v_j(x_n) \right)
  \quad \forall j \in J \colon a_j(x_n) < 0.
\end{equation}
Indeed, let
$$
  j(n) \in \argmin\bigg\{ f\left( x_n + \frac{1}{a_j(x_n)} v_j(x_n) \right) \biggm| j \in J \colon a_j(x_n) < 0 \bigg\},
$$
and denote $(a_n, v_n) = (a_{j(n)}(x_n), v_{j(n)}(x_n))$. Taking into account \eqref{MGCD_DecayOverOneIteration} it is
easy to see that there exists $n \le n^*$ such that either the method terminates at the $n$th iteration or $a_n < 0$
and $\| v_n \|^2 < \theta$, where $n^*$ and $\theta$ are defined in the proof of
Theorem~\ref{Thrm_MGCD_PiecewiseAffine}. Denote $y_n = x_n + a_n^{-1} v_n$. Arguing in the same way as in the proof of
Theorem~\ref{Thrm_MGCD_PiecewiseAffine} one can check that $a_{j(n)}(y_n) \ge 0$, which with the use of
Lemma~\ref{Lemma_MGCD_UselessVertices} and inequality \eqref{Codiff_FiniteConvergCond} implies that 
$a_{j(n)}(x_{n + 1}) \ge 0$. Furthermore, from \eqref{MGCD_DecayOverOneIteration} and \eqref{Codiff_FiniteConvergCond}
it follows that $f(x_{k + 1}) < f(x_k)$ for all $k \in \mathbb{N}$ such that $x_k$ is not a global minimizer of $f$.
Therefore, applying Lemma~\ref{Lemma_MGCD_UselessVertices} again one obtains that $a_{j(n)}(x_k) \ge 0$ for all 
$k \ge n + 1$. Repeating the same argument $s$ times one can easily verify that there exists $n \le s n^*$ such that
$a_j(x_k) \ge 0$ for all $k \ge n + 1$ and for all $j \in J$, which with the use of Theorem~\ref{Thrm_GlobOptCond}
implies the required result.

Observe that from \eqref{MGCD_GlobalDescentDirections} it follows that condition \eqref{Codiff_FiniteConvergCond} is
satisfied for the original version of method of codifferential descent with $\mu = + \infty$ and
$\overline{d}_{\mu} f(x) = \{ z_j(x) \mid j \in J \}$, which implies that the MCD also finds a point of global minimum
of a piecewise affine function in a finite number of steps.

Note that the MGCD discards those $(b_j, w_j)$ which no longer provide information about descent directions of
the function $f$, while the MCD keeps using all points $(b_j, w_j)$. Sometimes directions $v_j(x_n)$ such that
$a_j(x_n) \ge 0$ might provide some global information to the optimization method 
(i.e. $f(x_n) > \min_{\alpha > 0} f(x_n - \alpha v_j(x_n)$); however, this effect seems to be purely random, and it is
reasonable to discard those $j \in J$ for which $a_j(x_n) \ge 0$.

Let us finally note that it is unclear which version of the method of codifferential descent (the MCD or the MGCD) is
better for minimizing piecewise affine functions in terms of overall performance. Further research and extensive
numerical experiments are needed to answer this question. In particular, it is interesting to find a sharp upper
bound on the number of iterations of these methods. However, these questions lie outside the scope of this article, and
we leave them as open problems for future research.
\end{remark}

At the end of the paper, let us give a simple example demonstrating how one can compute a global codifferential mapping
of a piecewise affine function with the use of Proposition~\ref{Prp_GlobalCodiffCalc}, and how the MGCD can escape 
a local minimum, and find a point of global minimum in just one iteration.

\begin{example}
Let $d = 2$, and
$$
  f(x) = \min\big\{ \max\{|x_1|, |x_2| \}, 1 + \max\{ 2|x_1 - 2|, |x_2 - 2| \} \big\}.
$$
Set $x_0 = (2, 2)$. It is easily seen that $x_0$ is a point of local minimum of the function $f$, while a global
minimum is attained at the point $x^* = (0, 0)$. Our aim is to apply the MGCD with the starting point $x_0$ to the
function $f$. Instead of computing a DC decomposition of the function $f$ of the form \eqref{DCdecomp_PiecewiseAffFunc},
and then applying \eqref{DefOfGlobalCodiff_PiecewiseAff} in order to find $D f(x_0)$, we will compute a global
codifferential $D f(x_0)$ directly with the use of Proposition~\ref{Prp_GlobalCodiffCalc} (see
Remark~\ref{Remark_DCdecomp_Equiv_Codiff}). 

Let us compute $D f(x_0)$. Define
$$
  g_1(x) = \max\{|x_1|, |x_2| \}, \quad g_2(x) = 1 + \max\{ 2|x_1 - 2|, |x_2 - 2| \}.
$$
With the use of parts~\eqref{Calc_AffineFunc}, \eqref{Calc_PlusConstant} and (\ref{Calc_Max}) of
Proposition~\ref{Prp_GlobalCodiffCalc} one obtains
\begin{align*}
  \underline{d} g_1(x_0) &= \co\left\{ \begin{pmatrix} 0 \\ 1 \\ 0 \end{pmatrix}, 
  \begin{pmatrix} -4 \\ -1 \\ 0 \end{pmatrix}, \begin{pmatrix} 0 \\ 0 \\ 1 \end{pmatrix},
  \begin{pmatrix} -4 \\ 0 \\ -1 \end{pmatrix}
  \right\}, \quad \overline{d} g_1(x_0) = \{ 0 \}, \\
  \underline{d} g_2(x_0) &= \co\left\{ \begin{pmatrix} 0 \\ 2 \\ 0 \end{pmatrix}, 
  \begin{pmatrix} 0 \\ -2 \\ 0 \end{pmatrix}, \begin{pmatrix} 0 \\ 0 \\ 1 \end{pmatrix},
  \begin{pmatrix} 0 \\ 0 \\ -1 \end{pmatrix}
  \right\}, \quad \overline{d} g_2(x_0) = \{ 0 \}.
\end{align*}
Taking into account the fact that $f(x) = \min\{ g_1(x), g_2(x) \}$, and applying part~(\ref{Calc_Min}) of
Proposition~\ref{Prp_GlobalCodiffCalc} one gets that
$\underline{d} f(x_0) = \underline{d} g_1(x_0) + \underline{d} g_2(x_0)$, i.e.
\begin{align*}
  \underline{d} f(x_0) = \co\Biggl\{ &\begin{pmatrix} 0 \\ 3 \\ 0 \end{pmatrix},
  \begin{pmatrix} -4 \\ 1 \\ 0 \end{pmatrix}, \begin{pmatrix} 0 \\ 2 \\ 1 \end{pmatrix}, 
  \begin{pmatrix} -4 \\ 2 \\ -1 \end{pmatrix}, \begin{pmatrix} 0 \\ -1 \\ 0 \end{pmatrix}, 
  \begin{pmatrix} -4 \\ -3 \\ 0 \end{pmatrix}, \begin{pmatrix} 0 \\ -2 \\ 1 \end{pmatrix}, 
  \begin{pmatrix} -4 \\ -2 \\ -1 \end{pmatrix}, \\
  & \begin{pmatrix} 0 \\ 1 \\ 1 \end{pmatrix}, \begin{pmatrix} -4 \\ -1 \\ 1 \end{pmatrix}, 
  \begin{pmatrix} 0 \\ 0 \\ 2 \end{pmatrix}, \begin{pmatrix} -4 \\ 0 \\ 0 \end{pmatrix},
  \begin{pmatrix} 0 \\ 1 \\ -1 \end{pmatrix}, \begin{pmatrix} -4 \\ -1 \\ -1 \end{pmatrix},
  \begin{pmatrix} 0 \\ 0 \\ 0 \end{pmatrix}, \begin{pmatrix} -4 \\ 0 \\ -2 \end{pmatrix} \Biggr\}.
\end{align*}
One also has
\begin{align*}
  \overline{d} f(x_0) &= \co\left\{ \begin{pmatrix} 1 \\ 0 \\ 0 \end{pmatrix} 
  + \overline{d} g_1(x_0) - \underline{d} g_2(x_0),
  \overline{d} g_2(x_0) - \underline{d} g_1(x_0)
  \right\} \\
  &=\co\left\{ \begin{pmatrix} 1 \\ 2 \\ 0 \end{pmatrix}, \begin{pmatrix} 1 \\ -2 \\0 \end{pmatrix},
   \begin{pmatrix} 1 \\ 0 \\ 1 \end{pmatrix}, \begin{pmatrix} 1 \\ 0 \\ -1 \end{pmatrix},
   \begin{pmatrix} 0 \\ -1 \\0 \end{pmatrix}, \begin{pmatrix} 4 \\ 1 \\0 \end{pmatrix},
   \begin{pmatrix} 0 \\ 0 \\ -1 \end{pmatrix}, \begin{pmatrix} 4 \\ 0 \\ 1 \end{pmatrix}
   \right\}.
\end{align*}
Let us apply the MGCD. Solving the problem
$$
  \min \| (a, v) \|^2 \quad \text{subject to} \quad (a, v) \in \underline{d} f(x_0) + z_i(x_0)
$$
(Step~3 of the MGCD), one can check that for $z_1(x_0) = (1, 2, 0)^T \in \overline{d} f(x_0)$ one has
$$
  (a_1(x_0), v_1(x_0)) \approx ( -0.1111, 0.2222, 0.2222).
$$
Thus, $a_1(x_0) < 0$, and $x_0$ is not a point of global minimum of $f$ by Theorem~\ref{Thrm_GlobOptCond}. Furthermore,
one has $x_1 = x_0 + [a_1(x_0)]^{-1} v_1(x_0) = (0, 0) = x^*$, i.e. the MGCD finds a point of global minimum of
the function $f$ in just one step.
\end{example}

\section{Conclusions}

In this paper we analysed the performance of the method of codifferential descent in the case when the objective
function is either convex or piecewise affine. We proved that in the convex case this method has the iteration
complexity bound $\mathcal{O}(\varepsilon^{-1})$, provided the objective function satisfies some natural regularity
assumptions, which in the smooth case are reduced to the Lipschitz continuity of the gradient. We also proposed a
modification of the MCD for minimizing nonconvex piecewise affine function, and demonstrated that the modified method as
well as the MCD itself find a global minimizer of a nonconvex piecewise affine function in a finite number of steps. The
proof of this result is largely based on new global optimality conditions for piecewise affine functions obtained in
this article.

\section*{Acknowledgements}

The author wishes to express his thanks and gratitude to T. Angelov. The results of his numerical experiments, which he
shared with the author several years ago, demonstrated that the method of codifferential descent converges to a global
minimizer of a piecewise affine function in a finite number of steps. These results were the main source of 
the inspiration behind the research presented in this article.

\bibliographystyle{abbrv}  
\bibliography{Hypodiff_Descent}

\end{document}